\documentclass{amsart}

\usepackage[english]{babel}
\usepackage{csquotes}
\usepackage[style=alphabetic,
            isbn=false,
            doi=false,
            url=false,
            maxnames=10,
            maxalphanames=5
          ]{biblatex}
\renewbibmacro{in:}{}
\usepackage{amsmath,amsfonts,amsthm,mathtools, amssymb}
\usepackage[T1]{fontenc}
\usepackage{xcolor}
\usepackage{tikz-cd}
\usepackage{hyperref}
\hypersetup{hidelinks, colorlinks, citecolor=blue}
\usepackage{aliascnt}
\usepackage{xparse}
\usepackage{tensor}
\usepackage{caption}
\usepackage{mathrsfs}
\usepackage{tabto}
\usepackage{imakeidx}
\usepackage[normalem]{ulem}
\usepackage{upgreek}
\usepackage{stmaryrd}
\usepackage[all]{xy}
\usepackage[normalem]{ulem}
\usepackage{comment}

\bibliography{sources}

\usepackage{tikz}
\usetikzlibrary{matrix,arrows,calc}
\usetikzlibrary{positioning}
\usetikzlibrary{decorations.pathreplacing,decorations.markings,decorations.pathmorphing}
\usetikzlibrary{positioning,arrows,patterns}
\usetikzlibrary{cd}
\usetikzlibrary{intersections}

\tikzset{
	every loop/.style={very thick},
	comp/.style={circle,fill,black,,inner sep=0pt,minimum size=5pt},
	order bottom left/.style={pos=.05,left,font=\tiny},
	order top left/.style={pos=.9,left,font=\tiny},
	order bottom right/.style={pos=.05,right,font=\tiny},
	order top right/.style={pos=.9,right,font=\tiny},
	order node dis/.style={text width=.75cm},
	circled number/.style={circle, draw, inner sep=0pt, minimum size=12pt},
	below left with distance/.style={below left,text height=10pt},
    below right with distance/.style={below right,text height=10pt}
	}

\makeatletter
\ifcsname phantomsection\endcsname
    \newcommand*{\@gobblenexttocentry}[9]{}
\else
    \newcommand*{\@gobblenexttocentry}[4]{}
\fi
\newcommand*{\addsubsection}{%
    \addtocontents{toc}{\protect\@gobblenexttocentry}%
    \subsection*}
\makeatother

\begin{document}

\def\subsectionautorefname{Section}
\def\subsubsectionautorefname{Section}
\def\sectionautorefname{Section}
\def\equationautorefname~#1\null{(#1)\null}


\newcommand{\mynewtheorem}[4]{
  \if\relax\detokenize{#3}\relax 
    \if\relax\detokenize{#4}\relax 
      \newtheorem{#1}{#2}
    \else
      \newtheorem{#1}{#2}[#4]
    \fi
  \else
    \newaliascnt{#1}{#3}
    \newtheorem{#1}[#1]{#2}
    \aliascntresetthe{#1}
  \fi
  \expandafter\def\csname #1autorefname\endcsname{#2}
}

\newcommand{\mcomment}[1]{\textcolor{red}{#1}}
\newcommand{\dcomment}[1]{\textcolor{blue}{#1}}

\mynewtheorem{theorem}{Theorem}{}{section}
\mynewtheorem{lemma}{Lemma}{theorem}{}
\mynewtheorem{prop}{Proposition}{lemma}{}
\mynewtheorem{cor}{Corollary}{lemma}{}
\mynewtheorem{question}{Question}{lemma}{}
\mynewtheorem{assumption}{Assumption}{lemma}{}
\mynewtheorem{example}{Example}{lemma}{}

\theoremstyle{remark}
\mynewtheorem{rem}{Remark}{lemma}{}
\theoremstyle{definition}
\mynewtheorem{definition}{Definition}{lemma}{}
\mynewtheorem{reminder}{Reminder}{lemma}{}

\numberwithin{equation}{section}

\def\defbb#1{\expandafter\def\csname b#1\endcsname{\mathbb{#1}}}
\def\defcal#1{\expandafter\def\csname c#1\endcsname{\mathcal{#1}}}
\def\deffrak#1{\expandafter\def\csname frak#1\endcsname{\mathfrak{#1}}}
\def\defop#1{\expandafter\def\csname#1\endcsname{\operatorname{#1}}}
\def\defbf#1{\expandafter\def\csname b#1\endcsname{\mathbf{#1}}}

\makeatletter
\def\defcals#1{\@defcals#1\@nil}
\def\@defcals#1{\ifx#1\@nil\else\defcal{#1}\expandafter\@defcals\fi}
\def\deffraks#1{\@deffraks#1\@nil}
\def\@deffraks#1{\ifx#1\@nil\else\deffrak{#1}\expandafter\@deffraks\fi}
\def\defbbs#1{\@defbbs#1\@nil}
\def\@defbbs#1{\ifx#1\@nil\else\defbb{#1}\expandafter\@defbbs\fi}
\def\defbfs#1{\@defbfs#1\@nil}
\def\@defbfs#1{\ifx#1\@nil\else\defbf{#1}\expandafter\@defbfs\fi}
\def\defops#1{\@defops#1,\@nil}
\def\@defops#1,#2\@nil{\if\relax#1\relax\else\defop{#1}\fi\if\relax#2\relax\else\expandafter\@defops#2\@nil\fi}
\makeatother

\defbbs{ZHQCNPALRVWBEVWU}
\defcals{DOPQMNXYLTRAEHZKCFIUZG}
\deffraks{apijklmnopqueR}
\defops{PGL,SL,Sp,mod,Spec,Re,Gal,Tr,End,GL,Hom,PSL,H,div,Aut,rk,Mod,R,T,Tr,Mat,Vol,MV,Res,vol,Z,diag,Hyp,ord,Im,ev,U,dev,c,CH,fin,pr,Pic,lcm,ch,td,LG,id,Sym,Aut,PU}
\defbfs{kiuvzwp} 

\def\ep{\varepsilon}
\def\abs#1{\lvert#1\rvert}
\def\dd{\mathrm{d}}
\def\inj{\hookrightarrow}
\def\eq{=}
\newcommand{\hyp}{{\rm hyp}}
\newcommand{\odd}{{\rm odd}}

\def\i{\mathrm{i}}
\def\e{\mathrm{e}}
\def\st{\mathrm{st}}
\def\ct{\mathrm{ct}}

\def\uC{\underline{\bC}}
\def\ol{\overline}
  
\def\Vrel{\bV^{\mathrm{rel}}}
\def\Wrel{\bW^{\mathrm{rel}}}
\def\twolev{\mathrm{LG_1(B)}}

\def\be{\begin{equation}}   \def\ee{\end{equation}}     \def\bes{\begin{equation*}}    \def\ees{\end{equation*}}
\def\ba{\be\begin{aligned}} \def\ea{\end{aligned}\ee}   \def\bas{\bes\begin{aligned}}  \def\eas{\end{aligned}\ees}
\def\={\;=\;}  \def\+{\,+\,} \def\m{\,-\,}

\newcommand*{\proj}{\mathbb{P}}
\newcommand{\barmoduli}[1][g]{{\overline{\mathcal M}}_{#1}}
\newcommand{\moduli}[1][g]{{\mathcal M}_{#1}}
\newcommand{\omoduli}[1][g]{{\Omega\mathcal M}_{#1}}
\newcommand{\modulin}[1][g,n]{{\mathcal M}_{#1}}
\newcommand{\omodulin}[1][g,n]{{\Omega\mathcal M}_{#1}}
\newcommand{\zomoduli}[1][]{{\mathcal H}_{#1}}
\newcommand{\barzomoduli}[1][]{{\overline{\mathcal H}_{#1}}}
\newcommand{\pomoduli}[1][g]{{\proj\Omega\mathcal M}_{#1}}
\newcommand{\pomodulin}[1][g,n]{{\proj\Omega\mathcal M}_{#1}}
\newcommand{\pobarmoduli}[1][g]{{\proj\Omega\overline{\mathcal M}}_{#1}}
\newcommand{\pobarmodulin}[1][g,n]{{\proj\Omega\overline{\mathcal M}}_{#1}}
\newcommand{\potmoduli}[1][g]{\proj\Omega\tilde{\mathcal{M}}_{#1}}
\newcommand{\obarmoduli}[1][g]{{\Omega\overline{\mathcal M}}_{#1}}
\newcommand{\obarmodulio}[1][g]{{\Omega\overline{\mathcal M}}_{#1}^{0}}
\newcommand{\otmoduli}[1][g]{\Omega\tilde{\mathcal{M}}_{#1}}
\newcommand{\pom}[1][g]{\proj\Omega{\mathcal M}_{#1}}
\newcommand{\pobarm}[1][g]{\proj\Omega\overline{\mathcal M}_{#1}}
\newcommand{\pobarmn}[1][g,n]{\proj\Omega\overline{\mathcal M}_{#1}}
\newcommand{\princbound}{\partial\mathcal{H}}
\newcommand{\omoduliinc}[2][g,n]{{\Omega\mathcal M}_{#1}^{{\rm inc}}(#2)}
\newcommand{\obarmoduliinc}[2][g,n]{{\Omega\overline{\mathcal M}}_{#1}^{{\rm inc}}(#2)}
\newcommand{\pobarmoduliinc}[2][g,n]{{\proj\Omega\overline{\mathcal M}}_{#1}^{{\rm inc}}(#2)}
\newcommand{\otildemoduliinc}[2][g,n]{{\Omega\widetilde{\mathcal M}}_{#1}^{{\rm inc}}(#2)}
\newcommand{\potildemoduliinc}[2][g,n]{{\proj\Omega\widetilde{\mathcal M}}_{#1}^{{\rm inc}}(#2)}
\newcommand{\omoduliincp}[2][g,\lbrace n \rbrace]{{\Omega\mathcal M}_{#1}^{{\rm inc}}(#2)}
\newcommand{\obarmoduliincp}[2][g,\lbrace n \rbrace]{{\Omega\overline{\mathcal M}}_{#1}^{{\rm inc}}(#2)}
\newcommand{\obarmodulin}[1][g,n]{{\Omega\overline{\mathcal M}}_{#1}}
\newcommand{\LTH}[1][g,n]{{K \overline{\mathcal M}}_{#1}}
\newcommand{\PLS}[1][g,n]{{\bP\Xi \mathcal M}_{#1}}

\newcommand*{\can}{\mathrm{can}}
\newcommand*{\sing}{\mathrm{sing}}
\newcommand*{\reg}{\mathrm{reg}}
\newcommand*{\Ex}{\mathrm{Ex}}
\newcommand*{\Spin}{\mathrm{Spin}}
\newcommand*{\SU}{\mathrm{SU}}
\newcommand*{\sm}{\mathrm{sm}}
\newcommand*{\codim}{\mathrm{codim}}
\newcommand*{\Proj}{\mathrm{Proj}}
\newcommand*{\ad}{\mathrm{ad}}
\newcommand*{\Div}{\mathrm{Div}}
\newcommand*{\Ad}{\mathrm{Ad}}
\newcommand*{\SO}{\mathrm{SO}}

\newcommand*{\Tw}[1][\Lambda]{\mathrm{Tw}_{#1}}  
\newcommand*{\sTw}[1][\Lambda]{\mathrm{Tw}_{#1}^s}  

\newcommand{\bfa}{{\bf a}}
\newcommand{\bfb}{{\bf b}}
\newcommand{\bfd}{{\bf d}}
\newcommand{\bfe}{{\bf e}}
\newcommand{\bff}{{\bf f}}
\newcommand{\bfg}{{\bf g}}
\newcommand{\bfh}{{\bf h}}
\newcommand{\bfm}{{\bf m}}
\newcommand{\bfn}{{\bf n}}
\newcommand{\bfp}{{\bf p}}
\newcommand{\bfq}{{\bf q}}
\newcommand{\bfP}{{\bf P}}
\newcommand{\bfR}{{\bf R}}
\newcommand{\bfU}{{\bf U}}
\newcommand{\bfu}{{\bf u}}
\newcommand{\bfz}{{\bf z}}

\newcommand{\bfl}{{\boldsymbol{\ell}}}
\newcommand{\bfmu}{{\boldsymbol{\mu}}}
\newcommand{\bfeta}{{\boldsymbol{\eta}}}
\newcommand{\bfomega}{{\boldsymbol{\omega}}}

\newcommand{\wh}{\widehat}
\newcommand{\wt}{\widetilde}

\newcommand{\ps}{\mathrm{ps}}  
\newcommand{\gr}{\mathrm{gr}}  

\newcommand{\tdpm}[1][{\Gamma}]{\mathfrak{W}_{\operatorname{pm}}(#1)}
\newcommand{\tdps}[1][{\Gamma}]{\mathfrak{W}_{\operatorname{ps}}(#1)}

\newlength{\halfbls}\setlength{\halfbls}{.5\baselineskip}
\newlength{\halbls}\setlength{\halfbls}{.5\baselineskip}

\newcommand*{\Hrel}{\cH_{\text{rel}}^1}
\newcommand*{\Hrelbar}{\overline{\cH}^1_{\text{rel}}}

\newcommand*\interior[1]{\mathring{#1}}

\newcommand{\prodt}[1][\lceil j \rceil]{t_{#1}}
\newcommand{\prodtL}[1][\lceil L \rceil]{t_{#1}}



\newcommand{\remove}[1]{}
\newcommand{\sieg}{\mathbb{H}}
\newcommand{\Hg}{\operatorname{Hg}}
\newcommand{\tr}{\operatorname{Tr}}
\newcommand{\MT}{\operatorname{MT}}
\newcommand{\ra}{\rightarrow}
\def\erre{\mathbb{R}}
\newcommand{\C}{\mathbb{C}}
\newcommand{\Zeta}{{\mathbb{Z}}}
\def\esse{\mathbb{S}}
\newcommand{\A}{\mathcal{A}}
\newcommand{\rosso}[1]{{\color{red}#1}}
\newcommand{\blu}[1]{{\color{blue}#1}}

\title[Arakelov inequalities and characterization of ball quotients]
{Arakelov inequalities and characterization of totally geodesic ball quotients in $\mathcal{A}_g$}

\author{Matteo Costantini}
\address{Essener Seminar für Algebraische Geometrie und Arithmetik, Fakultät für Mathematik, Universit\"at Duisburg-Essen, 45117 Essen}
\email{matteo.costantini@uni-due.de}

\author{Daniel Greb}
\address{Essener Seminar für Algebraische Geometrie und Arithmetik, Fakultät für Mathematik, Universit\"at Duisburg-Essen, 45117 Essen}
\email{daniel.greb@uni-due.de}

\author{Carolina Tamborini}
\address{Essener Seminar für Algebraische Geometrie und Arithmetik, Fakultät für Mathematik, Universit\"at Duisburg-Essen, 45117 Essen}
\email{carolina.tamborini@uni-due.de}

\thanks{The research of the authors has been partially supported by the DFG-Research Training Group 2553 “Symmetries and classifying spaces: analytic, arithmetic, and derived”. During the preparation of this paper 
MC was supported by the DFG Eigene Stelle project  520961294, ``Die Geometrie der Strata von Differentialen”, while DG has been partially supported by the ANR-DFG project QuaSiDy. CT is a member of GNSAGA (INdAM)}

\date{\today}

\subjclass[2020]{}

\begin{abstract} We establish an Arakelov inequality for variations of Hodge structures underlying families of principally polarized complex Abelian varieties. In the equality case, this leads to a numerical characterization of certain totally geodesic ball quotients inside the moduli space of Abelian varieties.  Our result extends work of M\"oller--Viehweg--Zuo by removing the strong positivity conditions imposed in their statement. For families over compact base spaces, our proof involves showing that the period map associated with a family of Abelian varieties factors through certain MMP operations and then generalizing the results of M\"oller, Viehweg, and Zuo to an appropriate singular setting. In the quasiprojective surface case, we implement a new approach that does not pass through Miyaoka-Yau-type uniformisation theorems but uses an argument going back to an idea of Mok, symmetric space theory and semistability considerations instead.
\end{abstract}
\maketitle
\tableofcontents

\section{Introduction}

 In this paper we give an intrinsic numerical characterization of totally geodesic ball quotients inside the moduli space $\A_g$ of complex principally polarized Abelian varieties of dimension $g$. The moduli space $\A_g=\sieg_g/\operatorname{Sp}(2g, \mathbb{Z})$ is a quotient of the Siegel space $\sieg_g$ by the action of $\Sp(2g, \Zeta)$. Since the Siegel space $\sieg_g$ is a Riemannian and in fact Hermitian symmetric space, the moduli space $\A_g$ is endowed with a locally symmetric (orbifold) Riemannian metric. If $\pi: \sieg_g \rightarrow \A_g$ denotes the natural projection map, an algebraic subvariety $Z$ of $\A_g$ is called \emph{totally geodesic} if $Z=\pi(X)$ for some totally geodesic submanifold $X\subset \sieg_g$, which then turns out to be a Hermitian symmetric space as well. Totally geodesic subvarieties of $\A_g$ are closely related to \emph{special} (or \emph{Shimura}) subvarieties of the moduli space, which are defined as Hodge loci for the natural variation of Hodge structures on $\A_g$. In fact, {special} subvarieties of $\A_g$ are characterized as totally geodesic subvarieties that contain (countably many) complex multiplication points. 

In the following, we focus on the case where the uniformizing symmetric space $X$ of a totally geodesic subvariety is isomorphic to the complex unit ball $\mathbb{B}_d$.

\subsection{Main results}
The main result we prove here, \autoref{thm:intro1}, provides an intrinsic numerical characterisation of the situation in which the image of the period map $\varphi$ of a family of principally polarised Abelian varieties is a totally geodesic ball quotient inside $\cA_g$. 

It improves upon results obtained by  M\"oller-Viehweg-Zuo in \cite[Thm. 0.2]{MVZ}, which in turn are based on a number of prior advances that we discuss in detail in \autoref{sec:comments} below. The results in \emph{loc.~cit.} require strong positivity assumptions on a smooth snc compactification $(Y,D)$ of $U$, see \cite[Assumption 1.1]{MVZ}: $U$ is assumed to admit a smooth projective compactification $Y$ with normal crossing boundary divisor $D = Y \setminus U$ such that $\Omega^1(\log D)$ is nef and the log canonical bundle $\mathscr{O}_Y(K_Y + D)$ is ample with respect to $U$.  The main motivation for our work is to remove the smoothness and positivity assumptions on the log-canonical bundle, at the price of computing numerical invariants and working on a well-adapted potentially singular birational model of $(Y,D)$, and to discuss the assumption on the log-cotangent bundle in light of recent semistability results for such sheaves.  More precisely, we will work in the following general

\subsubsection{Setup}\label{subsec:setting} Let $U$ be an irreducible smooth quasi-projective variety. Assume either that $U$ is proper (and hence projective) of arbitrary dimension $d\geq 2$  or a quasi-projective surface. Let $f : A\to U$ be a family of principally polarised Abelian varieties of dimension $g$ and denote by $\varphi= \varphi_f :U \to \cA_g$ the associated period map.\footnote{Which we will without loss of generality assume to map to a \emph{fine} moduli space of polarised Abelian varieties.} Let $Y$ be a smooth projective compactification of $U$ with simple normal crossing divisor $D:=Y\setminus U$. Assume that 

  \begin{enumerate}
  \item $R^1f_*\underline{\bC}$ has unipotent local monodromies at infinity;
  \item $\varphi:U \to \cA_g$ is generically finite;
  \item the family $f: A \to U$ is non-extendable in the sense of  \autoref{def:nonextendabable}.
  \end{enumerate}
  \bigskip 
  The non-extendability condition  is only relevant in the quasi-projective case, i.e., when $D\neq \emptyset$, and it is discussed in detail in \autoref{subsect:nonextendable}.  Roughly speaking, it requires that the family $f$ (resp. the associated period map $\varphi_f$) cannot be extended to any of the irreducible components of the boundary $D$. This can always be achieved by extending a given family over as many components of $D = Y \setminus U$ as possible. Non-extendable families in particular have proper period maps $\varphi: U \to \mathcal{A}_g$. The reasons for the restriction to dimension two in the quasi-projective case are discussed in \autoref{sec:introperiodmapsandMMP} below. 
  
\subsubsection{Period maps and MMP operations}\label{sec:introperiodmapsandMMP}

In order to able to pass to the well-adapted birational model mentioned above, we first  prove that period maps factor through certain operations of the minimal model program. We summarise the results that we obtain in \autoref{sec:factorization} as follows:

\begin{theorem}[Factorization of period maps]\label{intro:fact} In the setup described in Section~\ref{subsec:setting}, there exist a birational rational map $g:~(Y, D) \dasharrow (Y_c, D_c)$ that does not extract a divisor and has the following properties:
\begin{enumerate}
    \item[(i)] $Y_c$ has canonical singularities if  $D=\emptyset$, and $(Y_c, D_c)$ is a log-canonical orbifold pair if $Y$ has dimension two and $D\neq \emptyset$.
    \item[(ii)] $K_c+D_c$ is nef, big, and ample on $U_c:=Y_c\setminus D_c$.\footnote{In the sense of \autoref{def:ampleonU}.}
    \item[(iii)] The period map $\varphi$ factors through a holomorphic map $\varphi_c: U_c\ra \cA_g$ such that the induced VHS $\mathbb{V}_c$ has unipotent local monodromy at infinity.
    \item[(iv)] The induced VHS $\mathbb{V}_c$ is non-extendable and its period map $\varphi_c: U_c\ra \cA_g$ is still proper. 
\end{enumerate}  
In the case where $D=\emptyset$, $Y_c$ is the \emph{canonical model} of $Y$.
\end{theorem}

The proof of part (iv) is closely related to the completeness question for the natural K\"ahler-Einstein metric on $Y \setminus D$, cf.~\cite{cadorel}. The results in \autoref{sec:factorization} are independent of \cite{MVZ} and their proofs use mainly detailed knowledge about the singularities of the minimal model program (both concerning their general structure in arbitrary dimension and their classification in the surface case) and fundamental facts about period maps. The main reasons we have to restrict to the surface case in the non-compact context (i.e., when $D \neq \emptyset$) is the lack of termination results for log-MMPs in dimensions $d\geq 4$ and our lack of detailed understanding of fibers of log-MMP operations in dimensions $d\geq 3$. 
\subsubsection{Characterisation of totally geodesic ball quotients}

Once we see that the families we are investigating really come from the partially ample model described above, the main goal is to characterize when the image of the period map $\varphi_c$ is a totally geodesic ball quotient. The following is our main result. 

\begin{theorem}[Characterization of totally geodesic ball quotients]
	\label{thm:intro1}
Consider the setup described in Section~\ref{subsec:setting} and let $f_c:A_c\to U_c$ be the family of Abelian varieties induced on the model $(Y_c,D_c)$ by \autoref{intro:fact}. Then, the image $\varphi(U)$ of the original period map is a totally geodesic ball quotient  if and only if the reflexive sheaf $\Sym^{[m]}\bigl(\Omega^1_{Y_C}(\log D_c)\bigr)$  is slope-stable with respect to $K_{Y_c} + D_c$ for all $m$, and if additionally there is a decomposition of the VHS $R^1{(f_c)}_* \underline{\bC}$  as
	\begin{equation}\label{eq:main}
	    R^1{(f_c)}_* \underline{\bC}=\bW\oplus \bU
	\end{equation}
	such that $\bU$ is a unitary local system and all $\mathbb{C}$-irreducible sub-VHS $\bV$ of $\bW$ fulfill the following conditions: 
	\begin{enumerate}
	    \item[(i)] 
	    The log-Higgs sheaf $(\mathscr{E}=\mathscr{E}^{1,0}\oplus \mathscr{E}^{0,1}, \theta)$  associated with $\bV$ satisfies the Arakelov equality
	    \[\mu(\mathscr{E}^{1,0}) -\mu(\mathscr{E}^{0,1}) = \mu\left(\Omega^1_{Y_{c}}(\log D_c)\right).\] 
	    \item[(ii)]  The length of the Higgs field $\zeta(\mathscr{E}, \theta)$ as defined in \autoref{def:lenght2} satisfies \[\zeta(\mathscr{E}, \theta)=\frac{\rk(\mathscr{E}^{1,0})\rk(\mathscr{E}^{0,1})(d+1)}{\rk(\mathscr{E})d}.\]
	    \end{enumerate}
\end{theorem}

We refer the reader to Section~\ref{subsect:MVZ-assumption} for a thorough discussion on the relation of the stability assumption on $\Sym^{[m]}\bigl(\Omega^1_{Y_C}(\log D_c)\bigr)$ made in our result to the condition imposed by M\"oller-Viehweg-Zuo. The necessary part of the statement is a special case of \cite[Prop.~0.1]{MVZ}.

\begin{rem}\label{rem:BMY_Yau} The non-projective case of \cite[Thm.~0.2]{MVZ} rests on results of Yau \cite{Yau93} that seem to not be correctly justified in \emph{loc.~cit.}, see the discussion in \cite[Sec.~1.2,(2)]{cadorel} and in \autoref{subsect:MVZ-assumption} below. In particular, it relies on a characterisation of quasi-projective ball quotients using the log-Bogomolov-Miyaoka-Yau equality that has been established only under restrictive assumptions not covering the setup considered in \cite{MVZ}. While in the surface case, the required uniformisation result is known thanks to work of Kobayashi-Nakamura-Sakai, see \autoref{thm:logMY} below, we nevertheless give a proof that circumvents the use of uniformisation via the BGM-equality and which therefore has a good chance of generalising to higher dimensions.  While heavily relying on some intermediate results in \cite{MVZ}, it provides a new argument also in the case of an snc surface pair,  avoiding certain additional problems discussed in \autoref{subsect:problemsinMVZ} below. 
\end{rem}

Examples of ball quotients realising the non-compact surface case and also others realising the higher-dimensional compact case are known to exist, e.g.~among unitary Shimura varieties of PEL type.

\subsection{Plan of the proof}

In order to prove the main result, after applying the factorization theorem, Theorem~\ref{intro:fact}, we need to analyze the Higgs bundles arising from the VHS also in the singular case.
The Arakelov inequality shown in \cite[Theorem 1]{VZ07}\footnote{See also {\cite[Thm. 6.4]{viehwegsurvey}} for a simplified proof.} can in fact (using new input on semistability of Higgs bundles on singular spaces) be shown to hold more generally on spaces with singularities as they occur in various steps of the minimal model program. In fact, in \autoref{sect:inequality} we establish the following very general result. 

\begin{theorem}[Arakelov inequality]
	\label{intro:inequality}
	Let $(Y,D)$ be a projective log-canonical pair of dimension $d$ such that  $K_Y+D$ is big and nef. 	Assume that $\bV$ is a non-unitary and irreducible $\mathbb{C}$-VHS of weight one on $Y\setminus D$ with unipotent local monodromies at infinity. Then, we have
	\begin{equation}\label{eq:Arakelov_intro}
	    \mu(\bV):=\mu(\mathscr{E}^{1,0}) -\mu(\mathscr{E}^{0,1})\leq \mu \bigl(\Omega_Y^{[1]}(\log D)\bigr) = (K_Y+D)^d,
	\end{equation}
	where $\mathscr{E}=(\mathscr{E}^{1,0}\oplus \mathscr{E}^{0,1},\theta)$ is the reflexive extension  of the system of log-Hodge bundles associated to $\bV$ from the snc locus of $(Y,D)$ to $Y$.
	
	Moreover, the equality $\mu(\bV)=\mu \left(\Omega_Y^1(\log D )\right)$ implies that both $\mathscr{E}^{1,0}$ and $\mathscr{E}^{0,1}$ are semistable with respect to $K_Y + D$ (as reflexive sheaves).
\end{theorem}

The proof of the equality case falls into two cases, the one of compact bases and the one of noncompact bases. 

In the compact case, we are able to reduce ourselves to the compact version of the main result of \cite{MVZ}
by establishing that under the conditions imposed in \autoref{thm:intro1} and the assumption that the period map factors through a fine moduli space, the canonical model of the original base of the family is actually smooth.  The problems mentioned in \autoref{rem:BMY_Yau} above only occur if the base is noncompact,  so that based on an observation about the necessity of the positivity assumption on the cotangent bundle imposed in \cite{MVZ}, see \autoref{prop:MVZnonef}, we are able to conclude the proof. 

For the non-compact case of \autoref{thm:intro1} we give a  self-contained proof in \autoref{sec:proof}. While the intermediate results are heavily inspired by some of the arguments of M\"oller-Viehweg-Zuo, the overall strategy is different. As before, we can show that $Y_c \setminus D_c = U_c$ is actually smooth. Then, however, we first use the length equality to show that the lift of the period map to the universal cover of $U_c$ has a totally geodesic submanifold inside Siegel space as its image. An argument going back to an idea of Mok  then shows that the period map is actually an embedding and a locally symmetric subspace.  In dimension two, this leaves us with two options: either this image is a ball or a bidisc. We exclude the second case by stability considerations; these involve comparing the Mumford compactification of the period map image of $U$ with the model $(Y_c, D_c)$ we built in the first step of our proof and rely on the existence of Zariski decompositions on surfaces. A comparison with the approach taken in \cite{MVZ} is given in \autoref{subsect:comparison_our_approach_with_MVZ}.

\subsection{History and earlier results}\label{sec:comments}

As regards \autoref{intro:inequality}, we note that the origin of Arakelov-type inequalities dates back to the work of Arakelov \cite{Arakelov}, who considered families of stable curves $f: \mathcal{C}\rightarrow Y$ over a smooth curve $Y$, smooth over $U=Y\setminus D$ for a finite subset $D\subset Y$. In this setting, Arakelov proved an inequality that can be expressed as
$$
\mu\left(f_*\omega_{\mathcal{C}/Y}\right)\leq \frac{1}{2}\, \mu\left(\Omega^1_Y(\log D)\right),
$$
linking the slope of the Hodge bundle of the family with the slope of the logarithmic cotangent bundle.

This inequality has been object of many generalizations. Here, we focus on families of Abelian varieties and, in particular, on the geometry of those families reaching the Arakelov equality. A first important generalization in this direction is due to Faltings \cite{Faltings}, who established an analogue of Arakelov’s inequality for families of Abelian varieties over curves.

Viehweg and Zuo \cite{VZ04} later characterized those families of Abelian varieties over curves whose image in $\A_g$ is a totally geodesic curve as precisely the ones attaining equality in the Arakelov inequality.
The inequality for family of Abelian varieties over higher-dimensional bases is established in \cite{VZ07}, while the characterization of the equality case is first partially addressed there and then fully treated in \cite{MVZ}.
\autoref{intro:inequality} generalizes the Arakelov inequality in \cite[Thm. 1]{VZ07} to the setting of log-canonical pairs of general type with nef log-canonical bundle. 

The reader is referred to Viehweg's survey article \cite{viehwegsurvey} for a discussion of the background,  a number of interesting aspects of the quantities used to characterise totally geodesic subvarieties in \cite{MVZ} and here, and the relation to other approaches such as Milnor-Wood-type  inequalities pursued for example in \cite{MWPaper}.

\subsection{Related results and further directions}

As recalled above, special subvarieties are totally geodesic, and among these they are characterized as those containing CM points. More precisely, totally geodesic subvarieties of $\A_g$ are linear translates of special subvarieties: Moonen \cite{Moonen} showed that any totally geodesic subvariety $Z$ that is not special is a fiber $Z \cong S_1 \times \{b\}$ of a product $S_1 \times S_2$ of special subvarieties, with $b \in S_2$ not a CM point (a fiber is special if and only if it lies over a CM point of $S_2$).  Applied to period map images of families $f : A \to U$ of principally polarized Abelian varieties, this result implies that if the family $f : A \to U$ is \emph{rigid}, i.e., if there are no non-trivial extensions of $f$ to a smooth family $\tilde f: \tilde A \to T \times U$ with $\dim T>0$, then being totally geodesic is equivalent to being Shimura. Consequently, under the rigidity assumption, our result gives an intrinsic characterisation of the situation in which the moduli image of $f$ is a special ball quotient.

The compact case of our main results (i.e., when $D$ is empty) is interesting from the point of view of recent work of Grushevsky, Mondello, Salvati Manni, and Tsimerman \cite{compact}, which gives sharp bounds on the dimension of compact subvarieties of $\mathcal{A}_g$. On the one hand, with growing $g$, compact subvarieties of higher and higher dimension exist in $\mathcal{A}_g$ so that finding out which of these are totally geodesic is a natural question, on the other hand, estimating the dimension of compact Shimura subvarieties in $\mathcal{A}_g$ is one of the major parts of the proof carried out in \cite{compact}.

Finally, let us point out a relation of Arakelov inequalities to the Coleman-Oort conjecture. Denoting by $\mathcal{M}_g$ the moduli space of smooth projective curves of genus $g$, and by $j:\mathcal{M}_g \to \mathcal{A}_g$ the corresponding period map, the conjecture predicts that for $g\geq 8$ there are no special subvarieties $Z \subset \mathcal{A}_g$ lying generically in the image of $j$, i.e., such that $Z \subset \overline{j(\mathcal{M}_g)}$ and $Z \cap j(\mathcal{M}_g) \neq \emptyset$. While we refer to \cite{Ghigi, MoonenOort} for a thorough survey on the conjecture, we just note that one possible approach to  attack the conjecture is to prove a strict Arakelov inequality for subvarieties of $\A_g$ that are generically contained in the Torelli locus. This strategy has already yielded several non-existence results for various types of Shimura curves in the Torelli locus e.g.~for Shimura curves with large, respectively small, unitary rank \cite{Chen-Lu-Zuo_2016}; see however the explicit examples constructed in \cite{andreas}.

Moreover, Moonen  \cite{moonen-threecases} showed that order to establish the Coleman–Oort conjecture it suffices to consider minimal special subvarieties 
and that, for $g\geq 8$, a minimal special subvariety $S$ of $\A_g$ which is generically contained in the Torelli locus must satisfy one of the following: either $\dim S = 1$, or $\dim S = 2$ and $S$ is compact, or $S$ is a compact ball quotient. So again, the case of compact ball quotients plays a decisive role. Furthermore, characterising compact locally symmetric surfaces not uniformised by the ball is an interesting open question.

\subsection{Acknowledgements}
The authors want to thank Beno\^{i}t Cadorel for enlightening discussions concerning uniformisation of quasi-projective varieties and his paper \cite{cadorel}. They also thank Alessandro Ghigi and Christian Miebach for interesting and helpful discussions on totally geodesic embeddings in Siegel space. Moreover, they thank Ben Moonen for his insightful remarks regarding special subvarieties of $\mathcal{A}_g$, Martin M\"oller for patiently answering our questions concerning \cite{MVZ}, Niklas M\"uller for discussions on the notion of log-Higgs sheaf, and Andreas Pieper for several interesting comments regarding families of Jacobians reaching the Arakelov bound. 

\section{Preliminary material}
\label{sec:background}

We work over the field $\mathbb{C}$ of complex numbers. 

\subsection{Preliminaries on VHSs and Higgs bundles on smooth varieties with snc compactification}\label{subsect:prelim_VHS} 
In this section, we recall some basic notions on variations of Hodge structures on smooth varieties, in particular the construction of the Higgs bundle associated with a VHS and the definition of its complexity, which is one of the numerical invariants appearing in our main statement. 

We start by recalling the fundamental notions surrounding families of Abelian varieties over quasi-projective base manifolds, as used in \cite{MVZ}. Let $U$ be an irreducible non-singular quasi-projective variety and $f : A\to U$ a family of principally polarized Abelian varieties of dimension $g$.  Equivalently, this corresponds to a polarized $\mathbb{Z}$-variation of Hodge structures of weight one $(\mathbb{W}_0, \mathscr{H}, \nabla, E^{1,0}, Q)$ over $U$. Here $\mathbb{W}_0=R^1f_*\underline{\mathbb{Z}}$ is a local system, $\mathscr{H}:= \mathbb{W}_0\otimes_{\mathbb{Z}} \mathscr{O}_U$ is the associated holomorphic vector bundle, $\nabla: \mathscr{H} \ra \mathscr{H}\otimes \Omega^1_U$ is the (flat) Gauss-Manin connection, $Q: \mathbb{W}_0 \otimes \mathbb{W}_0 \ra \underline{\mathbb{Z}}$ is the polarization, and $E^{1,0}$ is a holomorphic subbundle of $\mathscr{H}$ on which the Hodge metric $h(x,y)=i Q_{\mathbb{C}}(x, \overline{y})$ is positive definite. As a $\mathcal{C}^{\infty}$-vector bundle, $\mathscr{H}$ decomposes as $\mathscr{H}\cong E^{1,0}\oplus \overline{E^{1,0}} $. Set $E^{0,1}:= \mathscr{H}/E^{1,0}$, $E:= E^{1,0} \oplus E^{0,1}$,  and consider the following exact sequence of holomorphic bundles
\begin{gather*}
    0 \ra E^{1,0} \overset{i}{\longrightarrow} \mathscr{H} \overset{\pi}{\longrightarrow} E^{0,1} \ra 0.
\end{gather*}
The \emph{Higgs field associated with the VHS} is (by slight abuse of notation) the $\mathscr{O}_U$-linear map $\theta: E \to E \otimes \Omega^1_U$ induced by
\begin{gather*}
    \theta:= \pi \circ \nabla \circ i : E^{1,0} \longrightarrow E^{0,1} \otimes \Omega^1_U.
\end{gather*}
We call $(E, \theta)$  the \emph{Higgs bundle associated with the VHS}.  

Now let $Y$ be a smooth projective compactification of $U$ with smooth normal crossing divisor $D:=Y\setminus U$, and assume that $R^1f_*\underline{\bC}$ has unipotent local monodromies at infinity. Then $E^{1,0}$ extends to a subbundle of the Deligne extension of $\mathscr{H}$ to $Y$ and the Gauss-Manin connection extends to a logarithmic connection $   \nabla: \mathscr{H} \ra \mathscr{H} \otimes \Omega^1_Y(\log D)$ which induces a logarithmic Higgs field
\begin{gather*}
    \theta: E^{1,0} \longrightarrow E^{0,1} \otimes \Omega^1_Y(\log D)
\end{gather*}
that extends $\theta$ from $U$ to $Y$. Let $P^0:=\ker(\nabla\circ i)\subset E^{1,0}$ be the maximal unitary flat subbundle associated with the VHS. Then, $\mathbb{U}=P^0\oplus \overline{P^0}\subset \mathscr{H}$ defines a sub-$\mathbb{C}$-VHS, and hence Deligne's semisimplicity implies that the PVHS $R^1f_*\underline{\bC}$ decomposes as 
\begin{gather*}
   R^1f_*\underline{\bC}= \mathbb{U} \oplus \mathbb{W}.
\end{gather*}

 In the following, we consider irreducible $\mathbb{C}$-sub-local systems
$\mathbb{V} \subset \mathbb{W}$ and study numerical invariants attached to
them. Let $\mathbb{V} \subset \mathbb{W}$ be such an irreducible
$\mathbb{C}$-sub-local system, and denote by
$(E = E^{1,0} \oplus E^{0,1}, \theta)$ the Higgs bundle associated with
$\mathbb{V}$. We recall from \cite[Section~1.2]{MVZ} the
definition of the length $\zeta(\mathbb{V}) = \zeta(E,\theta)$, which
measures the complexity of the associated Higgs field.
For an integer $l > 0$, the Higgs bundle $(E, \theta)$ induces a Higgs
bundle
\begin{gather*}
 \bigwedge^{l}(E,\theta)=\left(\bigoplus_{m=0}^{l}E^{l-m,m},\bigoplus_{m=0}^{l-1} \theta_{l-m,}\right),   
\end{gather*}
where
\begin{gather*}
    E^{l-m,m}:=\bigwedge^{l-m} E^{1,0}\otimes\bigwedge^{m} E^{0,1},
\end{gather*}
and where the morphisms
\begin{gather*}
  \theta_{l-m,m}:E^{l-m,m}\longrightarrow E^{l-m-1,m+1}\otimes \Omega^1_Y(\log D)  
\end{gather*}
are induced by $\theta$. In the case where $l = \rk(E^{1,0})$ is maximal, we have that $E^{l,0} = \det(E^{1,0})$, and we write
\begin{gather*}
 \theta^{(m)}:=\theta_{l-m+1,m-1} \circ \cdots \circ \theta_{l,0} : \; \det(E^{1,0}) \longrightarrow E^{l-m,m} \otimes \Sym^m\bigl(\Omega^1(\log D)\bigr)
\end{gather*}
for the \emph{iterated Higgs field} or \emph{Griffiths-Yukawa coupling}. 
\begin{definition}[Length of a Higgs bundle, \protect{see \cite[after Eq.~0.3]{MVZ}}]\label{def:lenght}
The \emph{length of the Higgs bundle} $(E,\theta)$  is defined to be
\begin{equation}
\zeta(E,\theta):=\max\bigl\{m\in \bN  \mid  \theta^{(m)} \not\equiv 0 \bigr\}.
\end{equation}
Sometimes this is abbreviated to $\zeta(\mathbb{V})$, pointing to the fact that the Higgs bundle $(E, \theta)$ is induced by the VHS.
\end{definition}

Related to the previous discussion, let $\langle \det(E^{1,0}) \rangle \subset \bigwedge^{l}(E,\theta)$
denote the Higgs subsheaf generated by $\det(E^{1,0})$ inside the Higgs bundle $\bigwedge^{l}(E,\theta)$. The length is then equal to 
\begin{equation}
    \zeta(E,\theta) = \max\bigl\{m\in \bN  \mid   \langle \det(E^{1,0}) \rangle^{l-m, m} \neq \{0\}\bigr\}.
\end{equation}

\subsection{Discussion of technical assumptions} 

In this section, we will discuss the technical assumptions made in \cite{MVZ} and in which way these are dropped in the current paper. 

\subsubsection{The MVZ-assumptions}\label{subsect:MVZ-assumption}

The main result of M\"oller-Viehweg-Zuo \cite[Thm.~0.2]{MVZ} that we wish to generalize is formulated under strong positivity assumptions on the smooth snc compactification $(Y,D)$ of $U = Y \setminus D$ . Before we formulate these assumptions, let us recall the notion of ampleness with respect to an open subset.  

\begin{definition}\label{def:ampleonU}
A $\mathbb{Q}$-Cartier divisor $D$ on a normal projective variety $Y$ is called \emph{ample with respect to the Zariski-open subset $U\subset Y$} if for some sufficiently divisible $m\in \mathbb{N}_{\geq 1}$ that makes $mD$ Cartier, the stalks of $\mathscr{O}_Y(mD)$ at points of $U$ are generated by global sections, and the induced morphism $U \to \mathbb{P}\bigl(H^0\bigl(Y, \mathscr{O}_Y(mD)\bigr)^*\bigr)$ is a locally closed embedding. 
\end{definition}

The following is the set of assumptions that is used for most of \cite{MVZ}; after discussing these in the current section, we will discuss the setup of our main results in \autoref{subsubsect:our_assumptions}.

\begin{assumption}[{\cite[Assumptions 1.1]{MVZ}}]\label{assump:MVZ}
The logarithmic cotangent bundle $\Omega^1_Y(\log D)$ is nef, and the log canonical bundle $\omega_Y(D) = \mathscr{O}_Y(K_Y + D)$ is ample with respect to $U$ (in addition to being nef as a determinant of the nef bundle $\Omega^1_Y(\log D)$).
\end{assumption}

In order to explain the logic behind some of the arguments of \cite{MVZ} and in view of the fact that the result of Yau that forms one of the foundations of the arguments there seems to not be correctly justified in \cite{Yau93}, see \cite[Sec.~1.2,(2)]{cadorel},  in the current subsection we will additionally \emph{assume the validity of the statements made in \cite{Yau93}}. This  is equivalent to a completeness assumption for the natural K\"ahler-Einstein metric on $Y\setminus D$, which exists due to the ampleness assumption on $K_Y +D$ and which under the additional assumption can be used to derive de Rham-type decomposition results for the universal cover of $Y\setminus D$ that are used heavily in \cite{MVZ}, see Section 1.1 in \emph{loc.cit.}.

We first comment on the positivity assumptions on $\omega_Y(D)$. To begin with, nefness allows to define slopes $\mu(F)$ for coherent sheaves on $Y$ and hence the notions of $\mu$-(semi/poly)stability with respect to $\omega_Y(D)$.
Moreover, ampleness on $U$ and the existence of the complete K\"ahler-Einstein metric in the class $c_1(K_Y +D)$  allows to show that the logarithmic cotangent bundle $\Omega^1_Y(\log D)$ is slope-polystable with respect to $\omega_Y(D)$, even without assuming $\Omega^1_Y(\log D)$ to be nef. Consequently, we obtain a direct sum decomposition
\begin{gather}\label{eq:dec}
\Omega^1_Y(\log D)=\Omega_1\oplus \cdots \oplus \Omega_r,
\end{gather}
where each $\Omega_i$ is slope-stable of rank $n_i=\rk(\Omega_i)$. There are three cases for $\Omega_i$:
\begin{itemize}
\item \textbf{Type A:} $\Omega_i$ is invertible (i.e., $\rk(\Omega_i)=1$).  
\item \textbf{Type B:} $\rk(\Omega_i) > 1$, and for all $m>0$ the symmetric power $\mathrm{Sym}^m(\Omega_i)$ is slope-stable with respect to $\omega_Y(D)$. 
\item \textbf{Type C:} All remaining cases; that is, there exists some $m>1$ such that $\mathrm{Sym}^m(\Omega_i)$ is not slope-stable with respect to $\omega_Y(D)$.
\end{itemize}
Under the completeness assumption on the induced K\"ahler-Einstein metric, the decomposition \eqref{eq:dec} corresponds to a decomposition of the universal covering $\pi:~\tilde U \ra U$ as a product $\tilde U= M_1\times \ldots M_r$, where $\dim M_i=n_i$, cf.~\cite[Sect.~1.1]{MVZ}.

Second, we discuss the nefness assumption of $\Omega^1_Y(\log D)$. In fact, this implies that all the $\Omega_i$ and hence all the $\det(\Omega_i)$ are nef as well. Under the same completeness assumption for the natural K\"ahler-Einstein metric on $Y \setminus D$ one obtains the following geometric information for the factors $M_i$. 

\begin{itemize}
\item If $\Omega_i$ is of type A, then $M_i$ is a one-dimensional complex ball.  
\item If $\Omega_i$ is of type C, then $M_i$ is a bounded symmetric domain of rank greater than $1$.  
\item If $\Omega_i$ is of type B, then $M_i$ is an $n_i$-dimensional complex ball if and only if 
\begin{equation}\label{eq:BG}
[2 \cdot (n_i+1) \cdot c_2(\Omega_i) - n_i \cdot c_1(\Omega_i)^2] \cdot  c_1(\omega_Y(S))^{\dim Y - 2} = 0.
\end{equation}
\end{itemize}

Finally, let us point out that the assumption that $\omega_Y(D)$ is ample with respect to $U$ is also crucial for applying certain aspects of the quasi-projective Simpson correspondence. This is going to play a role also in our more general singular setting; cf.~\autoref{thm:simpsonVHS_singular}.

\subsubsection{Assumptions made in the current paper and an observation}\label{subsubsect:our_assumptions}

In the current paper,  will remove most of the MVZ positivity assumptions listed in \autoref{assump:MVZ} at the price of allowing some singularities that are inevitably introduced when passing to a well-adapted birational model $(Y_c, D_c)$. Here, the natural logarithmic tensor bundles appearing in the discussions above have to be replaced with their natural generalisation to singular spaces, which turn out to be reflexive sheaves of symmetric tensors that are customarily written as $\Sym^{[m]}\bigl( \Omega^1_{Y_c} \bigr)$. 

In order to single out Type B, which owing to the appearance of the Miyaoka-Yau-type condition \eqref{eq:BG} we consider the most interesting one, we  will assume that 
\begin{equation}\label{eq:stab_assump} \Sym^{[m]}(\Omega^1_{Y_c}(\log D_c)) \, \text{is stable for all $m \in \mathbb{N}_{\geq 1}$.}\end{equation}
In particular, this implies that $\Omega^{[1]}_{Y_c}(\log D_c)$ itself is stable so that only a single factor appears in the appropriate generalisation of \eqref{eq:dec}.

Once the logarithmic cotangent bundle is stable, that is, when only a single factor appears in \eqref{eq:dec}, the nefness assumption on $\Omega^1_Y(\log D)$ can be dropped. This can be deduced from a close reading of \cite{MVZ}, since the nefness of $\Omega^1_Y(\log D)$ is only used in \cite{MVZ} to deduce the nefness of $\det(\Omega_i)$ for all $i$. We capture this observation in the following proposition, which we state and will use only in the compact case $Y=U$, thus avoiding the complications mentioned in \autoref{rem:BMY_Yau} above.

\begin{prop}[Characterisation of compact smooth totally geodesic ball quotients]\label{prop:MVZnonef}
Let $Y$ be a smooth projective variety such that $\Omega^1_{Y}$ is stable and $\omega_Y$ is ample. Then the statements of \cite[Thm.~0.2]{MVZ} hold, that is: if $f : A\to Y$ is a family of principally polarized Abelian varieties of dimension $g$ such that the associated period map $\varphi:U \to \cA_g$ is generically finite, then the image $\varphi(U)$ is a totally geodesic ball quotient if and only if $\Sym^{m}\bigl(\Omega^1_{Y}\bigr)$  is slope-stable with respect to $K_{Y}$ for all $m$, and there is a decomposition of the VHS $R^1f_* \underline{\bC}$  as
	\begin{equation}
	    R^1f_* \underline{\bC}=\bW\oplus \bU,
	\end{equation}
	such that $\bU$ is a unitary local system and all irreducible components $\bV$ of $\bW$ fulfil the following conditions: 
	\begin{enumerate}
	    \item[(i)] $\bV$ satisfies Arakelov equality $\mu(E^{1,0}) -\mu(E^{0,1}) = \mu\left(\Omega^1_{Y}\right)$;
	    \item[(ii)] $\zeta(E, \theta)=\frac{\rk(E^{1,0})\rk(E^{0,1})(d+1)}{\rk(E)d}$
	\end{enumerate}
		    Here,  $(E, \theta)=(E^{1,0}\oplus E^{0,1}, \theta)$ is the system of log-Hodge bundles associated with the $\mathbb{C}$-VHS corresponding to $\mathbb{V}$, and $\zeta(E, \theta)$ is the complexity of its Higgs field.
\end{prop}

\begin{rem}

For a general connection between nefness and semistability of vector bundles see e.g.~\cite[Thm.~1.1]{RadloffJahnke} and \cite[Sect.~4a]{Nakayama}; in our setup, however, these considerations  do not lead to a direct implication between the two conditions, as it is not clear that stability will be preserved by restriction to arbitrary curves. 
\end{rem}

\subsection{Mumford compactification of totally geodesic subvarieties}\label{sec:Mumfordcompact}
Recall that the moduli space $\A_g$ is a quotient of the Siegel space $\sieg_g$ and denote by $\pi: \sieg_g\ra \A_g$ the projection map. 

\begin{definition}[Totally geodesic subvarieties]
An algebraic subvariety $Z \subset \A_g$ is called \emph{totally geodesic} if $Z=\pi(X)$ for some connected totally geodesic submanifold $X\subset \sieg_g$. We say that $X$ is the \emph{symmetric space uniformising $Z$}. 
\end{definition}
Totally geodesic submanifolds in $\sieg_g$ are themselves Hermitian symmetric space of non-compact type. In particular, such a submanifold $X$ is simply connected and, as every simply connected symmetric space, decomposes as the product 
\begin{gather}\label{eq:symmirr}
  X\cong X_1\times \dots \times X_s  
\end{gather}
of irreducible symmetric spaces, see e.g.~\cite[Cor.~6.12]{ziller}, the  classification of the irreducible pieces being classical, see e.g.~ \cite[Ch.~X]{Helgason}. 

The following proposition recalls from \cite[Section~3]{MVZ} some properties of the Mumford compactification of a smooth totally geodesic subvariety $Z$ of a fine moduli space $\A_g$ of principally polarised Abelian varieties with a suitable level structure, which we will use at the very end of our proof, see \autoref{subsect:noncompact_case}.

\begin{prop}[Properties of Mumford compactification, {\cite[Cor.~3.2]{MVZ}}]\label{prop:mumford}
Let $Z\subset A_g$ be a smooth totally geodesic subvariety of a fine moduli space $\A_g$ of principally polarised Abelian varieties with $N$-level structure with $N\geq 3$. Assume that the uniformising symmetric space $X$ of $Z$ decomposes as the product $X\cong X_1\times \dots \times X_s$ of irreducible symmetric spaces as in \eqref{eq:symmirr}. Let $\pi: X \ra Z$ be the projection map. The Mumford compactification  $(Y, D)$ of $Z\subset \A_g$ is a smooth snc compactification of $Z$ satisfying:
\begin{enumerate}
    \item the log-canonical bundle $\omega_{Y}(D)$ is big, nef, and ample with respect to $Z$;
    \item The log-cotangent $\Omega^1_{Y}(\log D)$ is polystable with respect to $\omega_{Y}(D)$.
    \item If $\Omega^1_{Y}(\log D)=\Omega_1\oplus \dots \oplus \Omega_r$ is the decomposition as a direct sum of stable sheaves, then $r=s$ and, for a suitable choice of the indices, $\pi^*(\Omega_i|_Z)=\mathrm{pr}_i^*\Omega^1_{X_i}$. 
\end{enumerate}
 In particular, if $Z$ is a surface and $X\cong \mathbb{B}_1 \times \mathbb{B}_1$, then $\Sym^{m}\Omega^1_{Y}(\log D)$ is not stable for all $m\geq 1$. In fact, for all $m\geq 1$ the line bundle $\mathscr{O}_{Y}\bigl(m(K_{Y}+D)\bigr)$ is a direct summand of $\Sym^{2m}\Omega^1_{Y}(\log D)$ that violates stability. 
\end{prop}
\begin{proof}
The points (1) to (3) are \cite[Cor~3.2]{MVZ}. Indeed, by \cite[Theorem 2.3]{MVZ}, since $Z$ is a smooth totally geodesic subvariety of $\A_g$, it is the base of a Kuga fiber space and thus the assumptions of \cite[Cor~3.2]{MVZ} are satisfied. As regards the last statement, if $X\cong X_1 \times X_2$ with $X_1\cong X_2 \cong  \mathbb{B}_1$, then point (3) implies that $\Omega^1_{Y}(\log D)=\Omega_1\oplus \Omega_2$ with $\pi^*(\Omega_i|_Z)=\mathrm{pr}_i^*\Omega^1_{X_i}$. Denoting by $\mu$ the slope with respect to $\omega_{Y}(D)$, we have that  $\mu(\Omega^1_{Y}(\log D))=\mu(\Omega_1)=\mu(\Omega_2)$, thus is it not stable.  More generally, for all $m\geq 1$ we have that $\mu(\Sym^m\Omega^1_{Y}(D))=m\mu(\Omega^1_{Y}(\log D))$ and, on the other hand, $\Sym^m\Omega^1_{Y}(D)=\oplus_{k=0}^m\Omega_1^k\otimes \Omega_2^{m-k}$ with each summand having slope $k\mu(\Omega_1)+(m-k)\mu(\Omega_2)=m\mu(\Omega^1_{Y}(\log D))$. Thus each summand ``destabilises'' the symmetric power. The last statement follows by observing that $\mathscr{O}_{Y}\bigl(m(K_{Y}+D)\bigr)=\Omega_1^m\otimes \Omega_2^{m}$, which appears as direct summand in $\Sym^{2m}\Omega^1_{Y}(D)$.
\end{proof}

	\subsection{Singularities and birational models}

In this section, we discuss the necessary notions and results from birational geometry that we will use. 

\subsubsection{Klt and lc pairs}

We will mainly work with the following types of singularities originating in the Minimal Model Program. The reader is refereed to \cite{kollarmori}, \cite{matsuki}, \cite{kollar} and \cite{Kawamata2024} for  in-depth discussions and examples. 
	\begin{definition}[Singularities of the MMP]\label{def:MMPsings}
		Let $X$ be a normal, quasi-projective variety and $D$ an effective Weil $\bQ$-divisor on $X$.   We say that $(X, D)$ is a \emph{pair}  if  all coefficients of $D$ are less than or equal to one and the Weil $\bQ$-divisor $K_X+D$ is $\bQ$-Cartier.
	
		We say that a pair $(X, D)$ is a \emph{klt pair} (resp. \emph{lc pair}) if \emph{for any} log resolution $\pi: \tilde X \rightarrow X$ of $(X, D)$, writing a $\bQ$-linear equivalence 		
		\begin{equation}\label{eq:discrepancy}
			K_{\tilde X} + \pi_*^{-1}(D) \;\sim_{\bQ} \; \pi^*(K_X+D)+ \sum_{i}a_i E_i,
		\end{equation}
		where $\pi_*^{-1}(D)$ denotes the strict transform and hence the $E_i\subset \tilde X$ are $\pi$-exceptional, the coefficients (called \emph{discrepancies}) $a_i\in\bQ$ satisfy $a_i>-1$, (resp. $a_i\geq-1$)  for all $i$.
		A normal, quasi-projective variety $X$ is said to be \emph{klt} (resp. lc) if $(X, 0)$ is klt (resp. lc).   A pair $(X,0)$ and then $X$ itself is called \emph{canonical / terminal} if all the discrepancies in any log resolution are at least $0$ / at least $1$.
		
		A pair $(X, D)$ is called \emph{dlt}, if \emph{there exists} a log resolution $\pi: \tilde X \rightarrow X$ such that the discrepancies $a_i$ appearing in the corresponding equivalence \eqref{eq:discrepancy} satisfy $a_i > -1$ for all $i$.
		
		\end{definition}
		
		\begin{rem}
		    We will mostly study pairs $(X, D)$ with reduced boundary $D = \sum_j D_j$; i.e., all the coefficients are equal to $1$.
			Minimal and canonical models of smooth projective varieties $X$ have canonical singularities. Minimal models of snc pairs $(X, D)$ are dlt pairs, while log canonical models of snc pairs of general type have lc singularities, see the discussion in Section~\ref{subsubsect:model} below. 
		\end{rem}

	\subsubsection{Orbifold pairs}\label{subsubsect:orbifold}
	
  Let $X$ be a normal quasi-projective variety. If the associated analytic space $X^{an}$ has \emph{finite quotient singularities}, i.e., is locally biholomorphic to a complex space of the form $\mathbb{C}^n /G$, where $G \subset \mathrm{GL}_n(\mathbb{C})$ is a finite subgroup without quasi-reflections, then the corresponding quotient map $q: \mathbb{C}^n \to \mathbb{C}^n/G$ branches only in codimension two or higher. Therefore, to $X$ one can associate in a unique way an (algebraic) Deligne-Mumford stack $\mathcal{X}$ with coarse moduli space $\mathcal{X} \to X$. 
	
	\begin{definition}[Orbifold pairs]\label{def:orbifold_pairs}
		We say that a pair $(X, D)$ is an \emph{orbifold pair} if $X^{an}$ has finite quotient singularities and if in addition $D$ defines an snc divisor on the corresponding smooth Deligne-Mumford stack $\mathcal{X}$.
	\end{definition}

	This notion for us gains its importance from the observation that, in addition to the classical facts that klt surface singularities are finite quotient singularities,  \cite[Prop.~4.18]{kollarmori}, and that dlt surface pairs are orbifold pairs, see e.g.~\cite[Sect.~1.13.6]{Kawamata2024},  many lc surface pairs are actually also orbifold. 
	
\begin{prop}\label{prop:lcpairs_as_quotients}
Let $p\in (X, D)$ be a (local) log canonical surface pair with reduced boundary divisor $D = \sum_{j=1}^n D_j$ and $(X, D)_{\sing} = \{p\}$. Then, $(X, D)$ is an orbifold pair. 
\end{prop}

This follows from the classification done in \cite[Thm.~9.6]{KawamataCrepant}, discussed also in \cite[Thm.~4-6-30]{matsuki}, and presented in a different way in \cite[Sect.~3.1]{Kobayashi_survey}; let us recall the main steps of the argument here. 
\begin{proof}[Sketch of proof of \autoref{prop:lcpairs_as_quotients}]
The point $p$ is a rational singularity by \cite[Prop.~2.28]{kollar} or \cite[Prop.~1.13.15]{Kawamata2024}.  Therefore, $X$ is $\mathbb{Q}$-factorial by \cite[Prop.~10.9]{kollar} or \cite[Prop.~1.13.16]{Kawamata2024}, so that dropping the boundary $D$ leads to $X$ having a klt and hence finite quotient singularity at $p$, see \cite[Lem.~9.2]{KawamataCrepant}. Locally, we hence get a finite quotient morphism 
\[
f: 0\in U \longrightarrow  X \ni p
\]
for a finite group acting linearly on $U \subset \mathbb{C}^2$
that is \'etale outside of $f^{-1}(p)=\{0\}$, and we consider the pullback pair $(U, f^{-1}(D))$, which is again log canonical e.g.~by \cite[Lem.~1.13.17]{Kawamata2024}, and therefore snc e.g.~by \cite[Prop.~2.42]{kollarmori}.
\end{proof}

       \subsubsection{(Log-)Canonical models of snc/klt varieties and pairs}  \label{subsubsect:model}
    
    A normal projective variety $X$ with $\mathbb{Q}$-Cartier canonical divisor $K_X$ is said to be \emph{of general type} if $K_X$ is big. Let $X$ be a klt variety of general type whose canonical divisor  $K_X$ is nef. Then, $K_X$ is semiample  by the Basepoint-Free Theorem \cite[Thm.~3.3]{kollarmori}; i.e., there exists a multiple of $K_X$ that is base-point free and hence defines a birational morphism  $
    q_X : X \to X_\can$     to a variety $X_\can$, called the \emph{canonical model}. The singularities of $X_\can$ are likewise klt, the $\mathbb{Q}$-Cartier divisor  $K_{X_\can}$ is ample, and 
    \begin{equation}K_X = q_X^* (K_{X_\can}), 
    \end{equation} see \cite[Lem.~2.30 or Prop.~3.51]{kollarmori}. 
    
   The following comparison result relating the fundamental groups of \( X \) and \( X_\can \) will be important in our study.
    
    \begin{theorem}[{\cite[Cor.~1.1(1)]{takayama}}]
    	\label{thm:takayama}
    	Let $X$ be a projective klt variety of general type with nef canonical divisor and let $q_X:X\to X_\can$ be the canonical morphism. Then the induced homomorphism of fundamental groups $q_{X*}:\pi_1(X)\to \pi_1(X_\can)$ is an isomorphism.
    \end{theorem}
    
    The situation in the case of nontrivial boundary at the outset looks similar; note however the shift in the singularity class when going to the log-canonical model: If the canonical divisor  $K_X+D$ of a projective dlt pair is big and nef,  then $K_X+D$ is semi-ample, again by the Basepoint-Free Theorem . As a consequence, we get a birational morphism $q: (X, D) \ra (X_{can}, D_{can})$, 
    where $(X_\can, D_\can)$ is log-canonical, $K_{X_{can}} + D_{can}$ is ample,  and 
    \begin{equation}
        K_X+D = q^* (K_{X_\can}+D_{can}).
    \end{equation} 
    
    In our study, pairs with $D \neq 0$ will appear via compactifications of quasi-projective base spaces of families of Abelian varieties. This implies that the relevant variations of Hodge structures will be defined only on $X \setminus D$ and not on the whole of $X$. This will make it necessary to study the geometry of the open parts $q^{-1}(y) \setminus D$ of fibres of log-canonical morphisms; see Section~\ref{sec:factorization} below.

    \subsection{Chern classes and stability }\label{subsect:stability}
Let $X$ be a normal projective variety of dimension $d$ with finite quotient singularities in codimension two; e.g., $X$ could be the underlying variety of a dlt pair $(X, D)$, see \autoref{subsubsect:orbifold}. Then, there exists a theory of first and second Chern classes for reflexive coherent sheaves $\mathscr{E}$ on $X$, developed in \cite[Sec.3]{GKPTens} and in more detail in \cite{Langer24}. Eventually, the idea goes back to \cite{MuEnum}. More precisely, if $\mathscr{E}$ and $\mathscr{F}$ are such sheaves there exist symmetric $\mathbb{Q}$-multilinear forms \begin{align*}
    \c_1(\cE):N^1(X)^{n-1}_{\bQ}\to \bQ,&\quad (L_1,\dots, L_{n-1})\mapsto c_1(\cE)\cdot L_1\cdots L_{n-1},\\
    \c_1(\cE)\c_1(\cF):N^1(X)^{n-2}_{\bQ}\to \bQ,&\quad (L_1,\dots, L_{n-2})\mapsto c_1(\cE)\c_1(\cF)\cdot L_1\cdots L_{n-2},\\
    c_2(\cE):N^1(X)^{n-2}_{\bQ}\to \bQ,&\quad (L_1,\dots, L_{n-2})\mapsto c_2(\cE)\cdot L_1\cdots L_{n-2},
\end{align*}
having a number of the usual properties, summarized for example in \cite[Section 2.2.1]{mueller} and derived in detail in \cite[Sect.~2.2.1]{muellerthesis}.

\begin{rem}\label{rem:deligneext}
If $X$  has finite quotient singularities everywhere, e.g.~if $X$ is the variety underlying an orbifold pair $(X,D)$, we say that $\mathcal{E}$ is an \emph{orbifold vector bundle} if it is induced by a vector bundle on the smooth DM-stack $\cX$ associated to $X$. Note that $H^i(X,\bQ)\cong H^i(\cX,\bQ)$, so we can define  orbifold $\mathbb{Q}$-Chern classes of orbifold vector bundles by considering the Chern classes of the associated vector bundle on the stack.  We will denote these classes by $\widehat{c}_i(\mathcal{E})$, $i= 1, \dots n$. Intersection theory on such spaces is equivalent to the intersection theory of smooth DM-stacks, which is classical, see e.g.~\cite{MuEnum}. For $i=1,2$, the associated $\mathbb{Q}$-multilinear forms coincide with the ones defined just above.
\end{rem}

Using first Chern classes, we can define degrees and slopes as usual.
  	 
\begin{definition}[Degree and slope] Let $H\in \Div(X)_\mathbb{Q}$ be a big and nef $\bQ$-Cartier $\mathbb{Q}$-divisor with numerical class $[H] \in N^1(X)_\mathbb{Q}$, which we call a \emph{polarization}. If $\mathscr{E}$ is a reflexive sheaf on $X$, the \emph{degree} and \emph{slope} of $E$ \emph{with respect to $H$} are defined as
\[\deg_H(\mathscr{E}):=\c_1(\mathscr{E})\cdot [H]^{n-1},\quad \mu_H(\mathscr{E}):=\frac{\deg_H(\mathscr{E})}{\rk(\mathscr{E})}.\]
 \end{definition} 
 
 Once degrees are defined, we can introduce the usual notions of slope-stability and slope-semistability, which have been employed extensively in higher-dimensional birational geometry.       
      
      \begin{definition}[Stability]\label{def:stability}
A torsion-free sheaf $\mathscr{E}$ on a normal projective variety $X$ is called \emph{stable} (respectively \emph{semistable}) with respect to a polarization $H$ if any  subsheaf $\mathscr{F}\subset \mathscr{E}$ with $0<\rk(\mathscr{F})<\rk(\mathscr{E})$ satisfies 
\begin{gather*}
	\mu_H(\mathscr{F})<\mu_H(\mathscr{E})\quad\quad (\text{respectively}\ 	\mu_H(\mathscr{F})\leq \mu_H(\mathscr{E})).
\end{gather*} 
  The sheaf $\cE$ is \emph{polystable} if it is a direct sum of stable sheaves of the same slope.
   \end{definition} 
   Next we discuss  generalisations of classical results  concerning semistability (for which the reader is referred to \cite{HL}) to singular spaces.   First, semistability puts restriction on morphisms, as the following precise result shows; the proof of the smooth case carries over to the singular setup. 
\begin{lemma}[{\cite[Addendum 5.2, v,ix]{MVZ}}]
	\label{lem:addendum}
Let $(X,D)$ be a pair and $H$ be a nef polarization ample on $U:=X\setminus D$. Let $\beta: \mathscr{G}\to \mathscr{F}$ be a morphism of two $\mu_H$-semistable sheaves of the same slope. Then the saturated image $\Im(\beta)_s$ is  a $\mu_H$-semistable subsheaf of $\mathscr{F}$ of slope $\mu_H(\mathscr{F})$  and
the inclusion $\Im(\beta)\to \Im(\beta)_s$  is an isomorphism on $U$.
\end{lemma}
   
The following two statements are much less formal; both can be deduced from known properties of semistability on singular spaces and with respect to slightly degenerate polarisations, however. 
\begin{prop}[Semistability is preserved under tensor operations]
	\label{thm:tensor_poly}
	Let $X$ be a normal variety, $H$ a  big and semiample divisor on $X$, and let $\mathscr{E}_1, \mathscr{E}_2$ be $\mu_H$-semistable reflexive sheaves.
	Then,  the reflexive tensor product \[\mathscr{E}_1 \left[\otimes \right] \mathscr{E}_2 := (\mathscr{E}_1\otimes \mathscr{E}_2)^{**} \] is $\mu_H$-semistable. Moreover, the same is true for $ \mathscr{H}\negthinspace om(\mathscr{E}_1, \mathscr{E}_2)$, as well as for reflexive symmetric and reflexive wedge products $\mathrm{Sym}^{[n]}\mathscr{E}_j$ and $\bigwedge^{[n]}\mathscr{E}_j$ of $\mathscr{E}_1$ and $\mathscr{E}_2$. 
\end{prop}
\begin{proof}
Using \cite[Thm.~6.1]{CompositioPaper} (with trivial Higgs fields and for sheaves not just defined on the regular locus but on the whole of $X$) iteratively, we may restrict all sheaves  occurring in the statement above to a complete intersection curve $C \subset X_{\reg}$ for $H$ along which these sheaves are locally free, preserving semistability. The claim then follows from classical results about vector bundles on the curve $C$, see e.g.~\cite{HL}.\footnote{Alternatively, we could have used the restriction theorem  \cite[Thm.~7.2]{LangerCrelle} in the first step.}
\end{proof}

\begin{prop}[Existence of Harder-Narasimhan filtration]
	\label{thm:HN}
	Let $Y$ be a normal projective variety  with a big and nef polarization $H$. Then $\mu_H$-Harder-Narasimhan filtrations of torsion-free, coherent sheaves of positive rank on $Y$  exist. 
\end{prop}

\begin{proof} If $|mH|$ is free, take a complete intersection curve $C = D_1 \cap \dots D_{n-1}$ of general elements $D_j \in |mH|$. The class of $C$ is movable in the sense of \cite[Def.~2.1]{GKPimrn} (and even big), and we may apply \cite[Cor.~2.27]{GKPimrn} to get the claimed statement, if $X$ is additionally assumed to be $\mathbb{Q}$-factorial. See \cite[Sect.~3]{LangerCrelle} for the general case, the case of ample polarisations being classical, see \cite[Thm.~1.6.7]{HL}
\end{proof}

We now establish a stability result that will play a key role in our discussion; it is a small upgrade of a fundamental result of Guenancia. 

\begin{theorem}[Semistability of the logarithmic tangent sheaf]
	\label{prop:guenanciass}
	Let $(X, D)$ be a log-canonical pair with $K_X+D$ big and nef.  Then the logarithmic tangent sheaf $\mathscr{T}_X(-\log D)$ is semistable with respect to $K_X+D$.
\end{theorem}
\begin{proof}
In \cite[Thm. A]{guenanciass} it is proven that if $K_X+D$ is ample, then $\mathscr{T}_X(-\log D)$ is polystable, and in particular semistable. We aim for a reduction to this result. 

If $K_X+D$ is only big and nef, then, as discussed in Section~\ref{subsubsect:model}, the canonical divisor $K_X + D =: H$ is semiample and induces a birational morphism to the log-canonical model $q:(X,D) \to (X_\can,D_\can)$, where $K_{X_\can}+D_\can =: H_\can$ is ample and  $K_X +D= q^*(K_{X_\can}+D_\can)$. Since $K_{X_\can}+D_\can$ is ample on $C_\can$, it is possible to find smooth complete intersection curves  curves $C \subset X$ for some multiple $m(K_X + D)$ and $C_\can \subset X_\can$ for $m(K_{X_\can}+D_\can)$ lying in the loci $V$ resp.~$q(V)$ on which $q: V \to q(V)$ is an isomorphism; in particular, $q$ induces an isomorphism $q:C\cong C_\can$. We will use these curves to compute the relevant degrees; without loss of generality $m=1$.

If $\mathscr{F}$ is a coherent subsheaf of $\mathscr{T}_X(-\log D)$, consider $\mathscr{F}_{|V}$. Since $q$ is an isomorphism over $V$, $\mathscr{F}_{|V}$ induces a coherent sheaf on $q(V)$, which extends to a coherent sheaf $\mathscr{F}_\can$ over $X_\can$ since $q(V)$ has complement of codimension at least two. Since  $\mathscr{F}_\can$ is a subsheaf of $\mathscr{T}_{X_\can}(-\log D_{\can})$ over $q(V)$, by reflexivity the same is true over the full $X_\can$.  Hence,  we can compute
\begin{align*}\mu_H(\mathscr{F})&=\deg_C(\mathscr{F})\\ &=\deg_{C_\can}\bigl(q_*(\mathscr{F})\bigr)\\
&\leq \mu_{H_\can}\bigl(\mathscr{T}_{X_\can}(-\log D_{\can})\bigr)\\
&=\mu_H\bigl(\mathscr{T}_{X}(-\log D)\bigr)\end{align*}
where the inequality comes from the semistability of $\mathscr{T}_{X_\can}(-\log D_{\can})$ proven by Guenancia, and the last equality again comes from the fact that $q$ is isomorphic along $C$.
\end{proof}
 
\subsection{Deligne extensions and log-Higgs sheaves}\label{sec:deligneext}

Let $(X,D)$ be a smooth quasi-projective snc pair and $\bV$ be a local system on the complement $X\setminus D$. We say that $\bV$ has \emph{unipotent local monodromy at infinity} if the monodromy of $\bV$ around the components of $D$ is unipotent. In this setup, by \cite[II, Sec. 5]{deligneeqdiff} (or \cite[Sec. 2]{schnell_can}) $\bV$ admits a canonical extension to $X$, the so-called \emph{Deligne extension}, which is a vector bundle on $X$ endowed with a connection with logarithmic poles along $D$. This is equivalent to $\bV$ having trivial parabolic structure in the sense of \cite[Condition 1.1 \& references given there]{mochizuki_asterisque}. 

\begin{rem}\label{rem:orbifoldDeligneExtension}
If $(X,D)$ is an orbifold pair and $\bV$ is an orbifold local system on $X\setminus D$, we can still define its orbifold Deligne extension to $X$ by defining it locally on \'etale charts on the associated stack,
 using the procedure known from the smooth case. 
 \end{rem}

\begin{prop}[Orbifold Chern classes of Deligne extensions]
\label{prop:deligneext}
Let $(X,D)$ be an orbifold pair and let $\bV$ be an orbifold local system on $X\setminus D$. Let $\bV_D$ be the orbifold Deligne extension of $\bV$ to $X$. Then \[\widehat{c}_i(\bV_D)=0 \in H^i(X,\bQ)\]  for all $i=1, \dots n$.
 \end{prop}
 \begin{proof}
 The result in the classical case of a smooth pair  $(X,D)$ follows from \cite[B3]{esnaultvielog}. Indeed, via Chern-Weil theory and by connecting the residue to the Atiyah class as in \cite[B1]{esnaultvielog}, it is possible to write the Chern classes of a sheaf with logarithmic connection in terms of symmetric polynomials evaluated on the residues of the logarithmic connection, see~\cite[Thm.~3, Thm.~5, Thm.~6]{Atiyah}. Moreover, for the Deligne extension all these terms are zero, since the assumption on unipotent local monodromy at infinity says that the residues are given by nilpotent matrices.
 The same argument applies in the orbifold case, the main technical point being that Chern-Weil theory still works in this setup, see e.g.~\cite{Satake}. See also the discussion around \cite[Theorem 17.5.21]{acgh2}.
 \end{proof}

While up to this point the discussion was restricted to local systems on orbifold pairs, we will need to discuss more general pairs, for which we need to define some of the basic notions. 
 
    \begin{definition}[Unipotent monodromy]\label{def:logHiggssing}
Let $(X,D)$ be a pair and $\bV$ be a local system on $U:=X\setminus D$.  We denote by $(X,D)_\reg = (X^\circ, D^\circ)$ the \emph{snc locus}  of the pair $(X,D)$.  We say that $\bV$ has \emph{unipotent local monodromies at infinity} if this holds for the restriction of $\bV$ to $U_{\reg} = U^\circ := X^\circ \setminus D^\circ \subset X^\circ$. 
\end{definition}

In the setup of the previous definition, we can  associate to the Deligne extension of $\bV|_{U^\circ}$ to $(X,D)_\reg$ a log-Higgs bundle $(E^\circ,\theta^\circ)$  on $(X,D)_\reg$ (see \cite[Sec. 13.2]{CM-SP}); here, $\theta^\circ:E^\circ\to E^\circ\otimes \Omega^1_{X^\circ}(\log D^\circ)$ is the log-Higgs field. By considering the reflexive extension to the whole of $X$, we obtain a so-called log-Higgs sheaf in the sense of the following definition, which goes back to Langer \cite{LangerCrelle}.

\begin{definition}[log-Higgs sheaves]
A \emph{log-Higgs sheaf} on a pair $(X,D)$ is a pair $(\mathscr{E}, \theta)$ consisting of a coherent sheaf $\mathscr{E}$ on $X$ and a morphism  
\[\theta: \mathscr{E} \otimes \mathscr{T}_X(- \log D) \to \mathscr{E}\] such that the morphism 
\[\bigwedge^2 \mathscr{T}_X(-\log D) \overset{[\cdot, \cdot ]}{\longrightarrow}  \mathscr{T}_X(-\log D) \overset{\theta}{\longrightarrow} \mathscr{E}nd(\mathscr{E}) \]
induced by $\theta$ and the Lie-bracket $[\cdot, \cdot ]$ vanishes. 
\end{definition}

A summary of the main properties of this extension of the concept of a Higgs bundle to the singular context is contained in \cite[Sect.~2.2]{mueller}, a detailed exposition and collection of technical background can be found in \cite[Sect.~2.3]{muellerthesis}. We will in particular discuss semistability of log-Higgs sheaves, which extends \autoref{def:stability} and use that a Higgs field on $\mathscr{E}$ induces natural Higgs fields on symmetric and alternating tensor powers of $\mathscr{E}$. 

\subsection{Length}
We now introduce the length of the log-Higgs sheaf $(\mathscr E,\theta)$. This notion generalises the one introduced in the smooth setting in \autoref{def:lenght}, and it is one of the main numerical invariants appearing in our main theorem, \autoref{thm:intro1}. As above, let $(X,D)$ be a pair and $\mathbb{V}$ a $\mathbb{C}$-irreducible local system on $U=X\setminus D$ with unipotent local monodromies at infinity, let $(E^\circ,\theta^\circ)$ be the associated log-Higgs bundle on $(X,D)_\reg$ and $(\mathscr{E}, \theta)$ its reflexive extension to $X$, i.e., $(\mathscr{E}=\mathscr{E}^{1,0}\oplus \mathscr{E}^{0,1}, \theta)$  is the log-Higgs sheaf associated with $\mathbb{V}$.

For an integer $l>0$, $(E^\circ, \theta^\circ)$ induces on $(X,D)_{\reg}$ a Higgs bundle $\bigwedge^l(E^\circ,\theta^\circ)$ as in \autoref{subsect:prelim_VHS}. In particular, in the case where $l=h:=\rk(E^{\circ\,1,0})$, we get on $(X,D)_{\reg}$ the associated iterated Higgs fields 
\begin{gather*}
 \theta^{\circ, (m)}=\theta^\circ_{l-m+1,m-1} \circ \cdots \circ \theta^\circ_{l,0} : \; \det\bigl((E^{\circ})^{1,0}\bigr) \to (E^{\circ})^ {l-m,m} \otimes \Sym^m\bigl(\Omega^1_{X^\circ}(\log D^\circ)\bigr)
\end{gather*}
or, equivalently, their dual morphisms
\begin{gather*}
 \det\bigl((E^{\circ})^{1,0}\bigr) \otimes \Sym^m \bigl( T_{X^\circ}(-\log D^{\circ})\bigr) \longrightarrow (E^{\circ})^ {l-m,m}. 
\end{gather*}
Taking reflexive extension to $X$ we get morphisms
\begin{equation*}
\theta^{(m)}: \det(\mathscr{E}^{1,0})[\otimes] \Sym^{[m]}\bigl(\mathscr{T}_X(-\log D)\bigr) \, \longrightarrow \, \bigwedge^{[h-m]}\mathscr{E}^{1,0}\,[\otimes]\, \bigwedge^{[m]}\mathscr{E}^{0,1}
\end{equation*}
We call $\theta^{(m)}$ the \emph{$m$-th iterated Higgs field} associated with $(\mathscr E,\theta)$. 

\begin{definition}[Length of a log-Higgs sheaf]\label{def:lenght2}
The \emph{length of the log-Higgs sheaf} $(\mathscr{E}, \theta)$  is defined to be
\begin{equation*}
\zeta(\mathscr{E}, \theta):=\max\bigl\{m\in \bN  \mid  \theta^{(m)} \not\equiv 0 \bigr\}.
\end{equation*}
\end{definition}
Equivalently, denoting by $\langle\det(\mathscr{E}^{1,0})\rangle
\subset
\bigwedge^{[h]}(\mathscr{E},\theta)$ the Higgs subsheaf generated by $\det(\mathscr{E}^{1,0})$, its 
$(h-m,m)$-piece coincides with $\langle\det(\mathscr{E}^{1,0})\rangle^{h-m,m}=\Im\bigl(\theta^{(m)}\bigr)$ 
and hence we have \begin{equation}\label{eq:length_equiv}
    \zeta(\mathscr{E}, \theta) = \max\bigl\{m\in \bN  \mid   \langle \det(\mathscr{E}^{1,0}) \rangle^{h-m,m} \neq \{0\}\bigr\}.
\end{equation}

	\subsection{Miyaoka-Yau (in)equality for orbifold log surfaces} 
	The following is the appropriate Miya\-oka-\-Yau inequality and uniformisation result in our setup; it is a consequence of the main result of \cite{MYlogsurfaces}, discussed also in \cite{Kobayashi_survey}; these results in fact apply to more general log-canonical surface pairs.
		 \begin{theorem}[Miyaoka-Yau for orbifold log surfaces, {\cite[Prop.~8.1]{CMS}}]
 		\label{thm:logMY}
  Let $(X, D)$ be projective orbifold surface pair such that $K_X+D$ is big and ample on $X\setminus D$. Then, identifying $H^4(X,\bQ)$ with $\mathbb{Q}$, the following inequality holds 
\begin{equation}\label{eq:MY}		3 \cdot \widehat{c}_2(\mathcal{T_X}(-\log D))- \widehat{c}_1(\mathcal{T_X}(-\log D))^2 \geq 0.
\end{equation}
Moreover, if equality holds in \eqref{eq:MY}, then $X\setminus D$ is a ball quotient; i.e.,  there is a cofinite
lattice $\Gamma\in \PU(1,2)$ such that $X\setminus D$ is biregular to $\bB^2/\Gamma$.
	\end{theorem}

	In \cite{MYlogsurfaces} and \cite[Thm.~2]{Kobayashi_survey} the algebraicity of the isomorphism $X\setminus D \cong \bB^2/\Gamma$ is in fact not explicitly discussed. This, however, is classical and follows from the fact that $\bB^2/\Gamma$ is hyperbolically embeddable in its Satake/Baily-Borel compactification combined with the $K^3$-theorem in the theory of complex hyperbolicity, see \cite[p.~279]{KobayashiHyperbolic}.

\section{Factorization of period maps via MMP}
\label{sec:factorization}

\subsection{Setup}\label{subsec:setupfactorization}
In this section we will work in the following setup. 
Let $U$ be an irreducible non-singular projective variety of dimension $d\geq 2$ or a quasi-projective surface and $f : A\to U$ a family of principally polarized Abelian varieties of dimension $g$. Denote by $\varphi= \varphi_f :U \to \cA_g$ the associated period map. Let $Y$ be a smooth projective compactification of $U$ with simple normal crossing divisor $D:=Y\setminus U$, so that  $U=Y$ if $U$ is projective. We assume

  \begin{enumerate}
  \item $R^1f_*\underline{\bC}$ has unipotent local monodromies at infinity;
  \item $\varphi:U \to \cA_g$ is generically finite. 
  \end{enumerate} 
 
We note that in the quasi-projective surface case this setup slightly differs from the global one stated in \autoref{subsec:setting}: in this subsection we are not (yet) assuming the non-extendability condition appearing there, which will be added in a second step in \autoref{subsect:nonextendable} below. 

\subsection{Bigness of the log-canonical bundle and existence of log-minimal model}

First of all we notice  that in this setup, the log-canonical bundle of $Y$ is big. 

\begin{theorem}[{\cite[Thm.~0.1(ii)]{zuobig}}]
	Let $Y$ be a smooth projective variety and $D\subset Y$ an snc divisor. Consider a polarised (integral) variation of Hodge structures on $U:=Y\setminus D$ with unipotent local monodromies at infinity. If the period map is locally injective at some point of $U$, then the log-canonical divisor $K_Y + D$ is big; i.e., $(X, D)$ is of log general type. 
\end{theorem}
\begin{rem}
Stronger results actually hold under less restrictive assumptions and are closely related to hyperbolicity questions for moduli spaces, see for example \cite{PopaSchnellhyp, Denghyp}, and in particular \cite{brunebarbe}.
\end{rem}

In order to be able to reach a situation where the log-canonical divisor is additionally nef (so that it defines a reasonable polarisation with respect to which one may consider semistability of sheaves as discussed in \autoref{subsect:stability} above), we run a minimal model program (with scaling), which terminates with a log-minimal model of $(Y, D)$. This is the first place at which a restriction on the dimension comes in when considering non-empty boundary divisors $D$.

\begin{prop}[Existence and singularities of minimal models] \label{prop:mmp}
Let $Y$ be a smooth projective variety and $D \subset Y$ a (reduced) simple normal crossing divisor such that $K_Y + D$ is big. If either $D = \emptyset$ or $\dim Y =2$, there exists a log-minimal model $(Y_{\min}, D_{\min})$ of $Y$; i.e., there exists a birational rational map $\psi: Y \dasharrow Y_{\min}$ that does not extract a divisor, $(Y_{\min}, D_{\min})$ is dlt, and $K_{Y_{\min}} + D_{\min}$ is nef and big. If $\dim Y =2$, the map $\psi$ is a (birational) morphism. 
\end{prop}
\begin{proof}
For the case of log-surfaces, this is proven in \cite{KollarKovacs} and also in \cite{Fujino_logsurfaces}, the statement concerning the singularities being known in arbitrary dimension for a long time, see \cite[Cor.~3.44]{kollarmori}. The existence of minimal models in the higher-dimensional case without boundary is contained in  \cite[Cor.~1.1.1]{BCHM}. See also \cite{Kawamata2024} for a thorough discussion of both the necessary results concerning termination of flips and the singularity categories of minimal models (of pairs). 
\end{proof}
\begin{rem}
Note that if $D = \emptyset$, the above theorem in particular states that $X_{\min}$ has klt singularities. In fact, although we will not need this fact, we note that the singularities in this case are even better, namely terminal, cf.~\autoref{def:MMPsings}.
\end{rem}

\subsection{Factorisation of period maps via minimal models}\label{subsect:minmodlfactor}
The first statement we want to show is that the period map we are starting with factors through the log-minimal model of $Y$.

\begin{prop}[Factorisation of period maps via minimal models]
	\label{prop:factorminimal}
In the setup of \autoref{subsec:setupfactorization}, let $g_{\min}: Y \dasharrow Y_{\min}$ be the (rational) map to a log-minimal model\footnote{As discussed in \autoref{prop:mmp}}, let $D_{\min}:=(g_{\min})_*D$ be the induced (reduced) divisor, and $U_{\min}:=Y_{\min}\setminus D_{\min}$. Then, the period map rationally factors through a regular map $\varphi_{\min}:U_{\min}\ra \cA_g$. More precisely, $\varphi$ is a regular extension of the rational map $\varphi_{\min} \circ (g_{\min})|_U$, so that
   \[\varphi = \varphi_{\min} \circ (g_{\min})|_U.\]
   Furthermore, $\varphi$ continues to be generically finite, and the induced VHS / the induced family of principally polarised Abelian varieties on $U_{\min}$ has unipotent local monodromies at infinity. 
\end{prop}
\begin{proof}
Recall that the minimal model is reached after a finite number of steps, through varieties $X^{(k)}$, $k =1, \dots, N$. In the surface case, each of the steps is a contraction \emph{morphism}, whereas in the higher-dimensional case, there are two types of steps, divisorial contractions and \emph{rational} flips that follow small contractions. 

We deal with the case $D=\emptyset$ first.  If $g:Y^{(k)} \to Y^{(k+1)}$ is a divisorial contraction, then the induced map $g_*: \pi_1\bigl(Y^{(k)} \bigr) \ra \pi_1\bigl(Y^{(k+1)}\bigr)$ is an isomorphism by \cite[Thm.~1.2]{takayama}. This implies that any local system on $Y^{(k)}$ is constant on the fibres of $g$. Factorisation of the period map via $g$ follows. If on the other hand
\[\xymatrix{   & \widetilde{Y} \ar[ld]_{\psi}\ar[rd]^{\pi}  &     \\
Y^{(k)} \ar[rd]_g \ar@{-->}[rr] &    & \ar[ld]  Y^{(k+1)} \\
         &  Y_{sm}  &    }
\]
is a flip diagram associated with the small contraction $g$, where $\psi: \widetilde{Y} \to Y^{(k)}$ is a resolution that resolves the rational map $Y^{(k)} \dasharrow Y^{(k+1)}$, then again by Takayama's result the pulled back period map $\varphi \circ \psi: \widetilde{Y} \to \mathcal{A}_g$ is constant along the fibres of  $\pi$ and hence descends to $Y^{(k+1)}$.

From now on, we deal with the case where $Y$ is $2$-dimensional and $D\neq \emptyset$. Again, we only need to consider one birational contraction morphism; to ease notation, we will denote it by $g: Y \to Y'$, with boundary divisor $D' = g_*(D)$.
To show that the period map factors though $U' = Y' \setminus D'$ we need to consider the curves $C\subset Y$ that are contracted by $g$ and have image in $Y' \setminus D'$. These occur in two types: 
\begin{enumerate}
\item Either $C$ is a component of a fibre of $g$ that does not intersect $D$, in particular $C \cap D = \emptyset$, or
\item $C$ is a connected component of the boundary divisor $D$.
\end{enumerate}
It follows from \cite[Lemma 2.3.5]{KollarKovacs} that in any case we have $C\cong \mathbb{P}^1$. If $C\cap D=\emptyset$, this immediately implies that the period map $\varphi$ is constant on $C$. As the corresponding fibre of $g$ is connected, we hence conclude that the period map descends to $Y'$.

It remains to consider case (2); so let $C$ be a connected component of $D$  that is contracted by $g$, so that $g(C)=:{p}\in U'$. By \autoref{prop:mmp}, we know that $p \in Y'$ is at worst klt, and therefore a finite quotient singularity by \cite[Prop.~4.18]{kollarmori}, cf.~the discussion in \autoref{subsubsect:orbifold}. Hence,  we may choose a small open neighborhood $V$ of $p$ in $U'$ such that $\bigl|\pi_1(V\setminus \{p\}) \bigr|$ is finite. By \cite[Theorem 9.5]{Griffiths1970}, the period map $\varphi' = \varphi \circ g|_{g^{-1}(V \setminus \{p\})}: V\setminus \{p\} \ra \cA_g $ extends to $p$. 

Finally, since the condition on local monodromy at infinity is checked at a general point of the boundary divisor $D_{\min} = (g_{\min})_*D$, the claim made in (3) follows from the assumptions fulfilled by the original family $f$. 

\end{proof}

\begin{rem}\label{rem:interiorsmooth}In the surface case, if $p \in U_{\min}$, the fibre of the the minimal model morphism (!) $Y \to Y_{\min}$ does not intersect the boundary divisor $D$. It follows that each contracted curve is actually a $(-1)$-curve and $U_{\min}$ is smooth. 
\end{rem}

\subsection{Factorisation of period maps via (partially) ample models}\label{subsect:canmodfactor}
In a second step, we want to descend the period map to models admitting a contraction from the (log) minimal model. First, we observe that the argument given at the beginning of the proof of \autoref{prop:factorminimal} together with Takayama's result recalled in \autoref{thm:takayama} establishes the following claim. 
\begin{prop}[Factorisation of period maps via canonical models]\label{prop:factorcanonicalcompact}
Let $Y_{\min}$ be the minimal model of a smooth projective variety $Y$ fulfilling the assumptions made in \autoref{subsec:setupfactorization}. Then, the period map $\varphi: Y_{\min} \to \mathcal{A}_g$ factors through the canonical morphism $g_{\can}: Y_{\min} \to Y_{\can}$. 
\end{prop}

If $D_{\min} \neq \emptyset$, the log canonical model is not the correct model to consider, for at least two closely related reasons. 
First, natural snc compactifications of ball quotients in general are not log canonical models; instead, their canonical divisors are nef, big, and ample on the complement of the boundary, see \cite[Sect.~3]{MVZ}; the (minimal) Baily-Borel compactification on the other hand has ample canonical, but has boundary of higher codimension.\footnote{The reader is also referred to the very interesting recent paper \cite{cadorel} of Cadorel for a discussion of the geometry of resolutions of ball quotient compactifications.} Secondly, concerning uniformisation theory, it follows from \autoref{thm:logMY} and the second point just mentioned that for an orbifold surface pair $(X,D)$ with $K_X + D$ ample equality in the log MY inequality actually forces the boundary to be empty, and hence eventually $Y$ as a whole to be mapped into $\mathcal{A}_g$. 

Therefore, we next construct a birational model of an snc surface pair that is better adapted to the setup of period maps and uniformisation problems than the log canonical model. Since our subsequent analysis of this models depends rather crucially on the fact that we are in dimension two, we decided to also describe the construction of $Y_c$ using classical surface theory.\footnote{See \cite[Sect.~1.4]{kollar} for a general construction of partial resolutions of log canonical models.}

\begin{prop}[Construction of an intermediate nef and partially ample model]\label{prop:factorcanonical}
Let $(Y_{\min}, D_{\min})$ be the minimal model of an snc surface pair $(Y,D)$ that fulfils the assumptions made in \autoref{subsec:setupfactorization}. Let $g_\can: (Y_{\min}, D_{\min}) \to (Y_{\can}, D_{\can})$ be the log canonical morphism. Let 
\begin{equation}
    \mathcal{D}_0 := \left\{B \subset Y_{\min} \text{ prime divisor }\; \mid\; g_{\can}(B) = \{pt.\} \right\}
\end{equation}  Then, there exist a birational contraction $g_c: (Y_{\min}, D_{\min}) \ra (Y_c, D_c)$ with the following properties:
\begin{enumerate}
    \item The pair $(Y_c, D_c)$ is a log-canonical projective orbifold, and $g_c$ contracts exactly all divisors $B \in \mathcal{D}_0$ with $B\cap U_{\min}\neq \emptyset$. In particular, $K_c + D_c$ is nef, big,  and ample on $U_c$.\footnote{In the sense of Definition~\ref{def:ampleonU}.} 
    \item The period $\varphi_{\min}$ factors through $\varphi_c: U_c\ra \cA_g$; in particular, $\varphi_c$ is generically finite.
    \item The induced VHS on $U_{\c}$ has unipotent local monodromies at infinity. 
\end{enumerate}
\end{prop}
\begin{proof} 
The union of exceptional fibres of $g_{\can}$ intersecting $U_{\min}$ contains exactly the curves we are interested in.  Since $Y_{\min}$ is $\mathbb{Q}$-factorial, it follows for example from \cite[Cor.~5.3]{schroeer} that there exists a proper normal algebraic surface $Y_c$ and a birational morphism $g_c: Y_{\min}\ra Y_c$ contracting exactly these fibres. 
Letting $D_c:=(g_c)_*(D_{\min})$, we have the following diagram:
\begin{equation}\label{eq:contractiondiagram}
  \begin{tikzcd}
    (Y_{\min}, D_{\min}) \arrow{r}{g_c} \arrow[swap]{dr}{g_{\can}} & (Y_c, D_c) \arrow{d}{g} \\
     & (Y_{\can}, D_{\can}).
  \end{tikzcd}
\end{equation}

In order to prove (1), we start analysing the singularities of the pair $(Y_c, D_c)$. Locally, near points $p \in Y_c$ with nontrivial exceptional fibre $g_c^{-1}(p)\subset Y_{\min}$, the singularities are the same as the singularities of the corresponding point in the log canonical model $(Y_{\can}, D_{\can})$. Hence, if $p \notin D_c$, the variety $Y_c$ has canonical singularities near $p$, see for example \cite[Thm.~1-6-4]{matsuki} for the necessary local computation, and if $p\in D_c$, the pair $(Y_c, D_c)$ is log canonical near $p$. We conclude from rationality of canonical singularities and \cite[Prop.~2.28]{kollar}, respectively, that $Y_c$ hat rational surface singularities. In particular, $Y_c$ is projective by \cite[Thm.~2.3]{ArtinNumericalCrit}. The classical fact that canonical surface singularities are finite quotient singularities and the orbifold property of log canonical surface pairs with reduced boundary divisor recalled in \autoref{prop:lcpairs_as_quotients} together imply that $(Y_c, D_c)$ is an orbifold pair. 

Moreover, looking at Diagram \autoref{eq:contractiondiagram}, recalling that $g_c$ contracts exactly those divisors of $\mathcal{D}_0$ that intersect $U_{\min}$, and using that $K_{Y_{\min}} + D_{\min}$ is the pullback of the ample divisor $K_{Y_{\can}} + D_{\can}$ under $g$, we see that 
\begin{equation}\label{eq:crepant}K_{Y_c}+ D_c = g^*(K_{Y_{\can}} + D_{\can})\end{equation}
is ample with respect to $U_c$, as claimed.

Starting to prove (2), we first notice that we only have to define the map providing the factorisation of the period map at points in $p\in U_c = Y_c \setminus D_c$ over which $g_c$ has positive-dimensional fibre. By construction, the fibre $g_c^{-1}(p)$ over such a point is completely contained in $U_{\min}$, which is smooth by \autoref{rem:interiorsmooth}; hence, $g_c: U_{\min} \to U_c$ is a minimal resolution. As $p \in Y_c$ is a rational and even canonical singularity, it follows either by using $H^1\bigl(g_c^{-1}(p), \mathscr{O} \bigr) = \{0\}$ or from the classification given e.g.~in \cite[Thm.~1.13.9]{Kawamata2024} that each such fibre is a tree of rational curves.  Consequently, the period map $\varphi_{\min}$ restricted to $g_c^{-1}(p)$ is constant. The desired factorisation follows.

Since the condition on monodromy at infinity is checked at a general point of the boundary divisor $D_{c} = (g_{c})_*D$, the claim made in (3) follows from the assumptions fulfilled by the original family $f$. 
\end{proof}

\subsection{Geometry of MMP and non-extendable period maps}\label{subsect:nonextendable}

Next, we discuss extendability of families, properness of period maps, and how these conditions interact with birational contractions of surface pairs.

\begin{definition}[Non-extendable families]\label{def:nonextendabable}
Let $f: A \to U = Y \setminus D \subset Y$ be as in \autoref{subsec:setupfactorization} and $\dim U =2$ with associated period map $\varphi_f :  U \to \mathcal{A}_g$ to some fine moduli space of Abelian varieties. We say that the family $f$ / the period map $\varphi_f$ is \emph{non-extendable} if the monodromy of $R^1f_*\mathbb{C}_A$ around any component of $D$ is infinite. 
\end{definition}

In fact, any family as in \autoref{subsec:setupfactorization} can be extended to a non-extendable family over those components of $D$ along which the monodromy is not infinite by \cite[Thm.~6.4]{GriffithsBulletin}.  Non-extendability can be characterized in terms of the associated period map as follows: A family $f: A \to U =Y\setminus D$ as above allows for an extension of its period map $\varphi_f:U \to \cA_g$ to a map from $Y$ to the the Baily-Borel  compactification of $\cA_g$, which we denote by $\overline{\varphi}_f:Y \to \bar{\cA}_g$, see \cite[Thm.~6.6]{GriffithsBulletin}.\footnote{As a consequence, the period map $\varphi_f$ is algebraic.} Then, if $f$ is assumed to be non-extendable, we have 
\begin{equation}
   \overline{ \varphi}_f (D) \, \subset \, \overline{\mathcal{A}}_g \setminus \mathcal{A}_g.
\end{equation}
I.e., the boundary of $U$ in Y is mapped to the boundary of $\mathcal{A}_g$ in the chosen compactification. As an important consequence, the period map $\varphi_f: U \to \mathcal{A}_g$ is proper, cf.~\cite[pp.~157f]{Griffiths1970}.

\begin{prop}[Non-extendable families and minimal models]\label{prop:propernesstominimal}
	Let $Y$ be a smooth projective surface and $D\subset Y$ a simple normal crossing divisor. Suppose that $U=Y\setminus D$ supports a non-extendable family of principally polarised Abelian varieties satisfying the assumption of \autoref{subsec:setting}. Let $C \subset Y$ be a positive-dimensional fibre of the  morphism  $g_{\min}: Y\ra Y_{\min}$ to the log minimal model. If $C \cap U \neq \emptyset$, then $C \cap D = \emptyset$, and hence $C \subset U$. 
	
	In other words, $U$ is saturated with respect to $g_{\min}$, and consequently, the descended period map $\varphi_{\min}: U_{\min} \to \mathcal{A}_g$ guaranteed by \autoref{prop:factorminimal} continues to be proper. 
\end{prop} 
\begin{proof} Assume that $C \cap U \neq \emptyset$ but (aiming for a contradiction) also $C \cap D \neq \emptyset$. We may assume that $g_{\min}$ is a contraction morphism in one step of the log minimal model program, so that $C = g_{\min}^{-1}(p)$ is irreducible. Applying \cite[Lemma 2.3.5]{KollarKovacs} we see that $C \cong \mathbb{P}^1$ and that  the image point $p \in (Y_{\min}, D_{\min})$ is plt. This implies that at most one branch of $D_{\min} = (g_{\min})_*(D)$ passes through $p$, see \cite[Thm.~4.15]{kollarmori} or \cite[start of Sect.~3.35]{kollar}. Therefore, $C$ intersects $D$ in a single point, say $q \in Y$. It follows that the period map $\varphi_f$ is constant on $C \setminus D \cong \mathbb{C}$, say with value $\xi \in \mathcal{A}_g$. Taking limits, it follows that $\overline{\varphi}_f (q) = \xi$.  This contradicts non-extendability of $f$, which says that $q\in D$ needs to be mapped to  $\overline{\mathcal{A}}_g \setminus \mathcal{A}_g$ by $\overline{\varphi}_f$. 
\end{proof}

Passing one step further to the log canonical model, the analogous statement is true; note that the proof is slightly more involved, as a priori we have to deal with full exceptional fibres of the log canonical morphism at once. 
\begin{prop}[Non-extendable families and partially ample models]\label{prop:propernesstomodel}
	Let $Y$ be a smooth projective surface and $D\subset Y$ a simple normal crossing divisor. Suppose that $U=Y\setminus D$ supports a non-extendable family of principally polarised Abelian varieties satisfying the assumption of \autoref{subsec:setting}. Let $C \subset Y$ be a positive-dimensional fibre of the  morphism  $\gamma: Y\ra Y_{\can}$ obtained by composing the map $g_{\min}: Y \to Y_{\min}$ to the log minimal model with the log canonical morphism $g_{\can}: (Y_{\min}, D_{\min}) \to (Y_{\can}, D_{\can})$. If $C \cap U \neq \emptyset$, then $C \cap D = \emptyset$, and hence $C \subset U$.\\
		In particular, $U_{\min}$ is saturated with respect to $g_{\can}$ and hence with respect to the contraction $g_c: Y_{\min} \to Y_c$ constructed in  \autoref{prop:factorcanonical} . As a consequence, the descended period map $\varphi_{c}: U_c \to \mathcal{A}_g$ continues to be proper. 
\end{prop}

\begin{proof}
Assume that for the fibre $C$ over a point $p \in Y_{\can}$, we have $C \cap U \neq \emptyset$ but (aiming for a contradiction) also $C \cap D \neq \emptyset$. 
As we have noticed before, since necessarily $p \in D_{\can}$, the variety $Y_{\can}$ underlying log canonical model $(Y_{\can}, D_{\can})$ has klt and hence rational singularities near $p$, cf.~the proof of \autoref{prop:lcpairs_as_quotients}. As a consequence, $C$ is a tree of rational curves. 

If one, say $C_0$, of these rational curves lies in $D$, let $C' \subset C$ be the largest connected union of irreducible components of $C$ containing $C_0$ and contained in $D$. This is a subtree of $C$. Since $C' \subset C$, and $C$ is contracted, the curve $C'$ has negative definite intersection matrix, and by Grauert's criterion can therefore be contracted to a point $q$ in a normal complex surface $Z$, say by the birational morphism $\psi: Y \to Z$; see also \cite{ArtinNumericalCrit}. If $D_Z := \psi_*(D)$, then by construction we have $q \in Z \setminus D_Z$.  Since $q$ is obtained by contracting the subgraph of the exceptional curve $C$ in the resolution $Y \to Y_{\can}$ of the klt singularity $p \in Y_{\can}$, the germ $(Z, q)$ also has klt singularities, as a short discrepancy computation shows. In particular, $Q$ is a finite quotient singularity, \cite[Prop.~4.18]{kollarmori}, and the local fundamental group of $Z$ at $q$ is finite. We may therefore apply \cite[Theorem 9.5]{Griffiths1970} to extend the restriction of the period map $\varphi_f$ to a neighbourhood of $C'$ over $q$, and hence over $C'$, contradicting non-extendability. 

By the previous considerations, we know that every component of $C$ intersects $U$. It follows that the number of points of intersection of $C$ with $D$ is estimated by the number of branches of $D_{\can}$ near $p$. The classification of log canonical surface pairs with reduced boundary divisor, \cite[Thm.~4.15]{kollarmori}, implies that there are at most two branches of $D_{\can}$ through $p$. Therefore, the period map $\varphi$ is constant along $C \setminus D$, which is a chain of rational curves with at most two points removed. The desired contradiction now follows as at the end of the proof of \autoref{prop:propernesstominimal} above. 
\end{proof}

\begin{rem}[Recent work of Cadorel]
If a fibre of the log canonical morphism $Y_{\min} \to Y_{can}$ does not intersect $U_{\min}$, it can be an elliptic curve, leading to infinite local fundamental group at the image point. The separation between boundary components and fibres of MMP operations observed above should be compared with the assumptions of Theorem 2 in Cadorel's recent paper \cite{cadorel}. In the same paper, he obtains very interesting results concerning the triviality of variations of Hodge structures along fibres of not necessarily isolated log canonical singularities. A generalisation to the higher-dimensional \emph{pair} situation would be an important step in extending the results obtained in our paper to families over higher-dimensional \emph{quasi-projective} base manifolds. 
\end{rem}

\section{The Arakelov inequality} \label{sect:inequality}

In this section we will prove \autoref{intro:inequality}. Let $(Y,D)$ be a pair and $\bV$ be a variation of Hodge structures on $U := Y\setminus D$ with unipotent local monodromies at infinity, i.e., around $D$.
Recall from \autoref{def:logHiggssing} that we can define the associated system of log-Hodge bundles $(E,\theta)$ on $Y$.

The following is a generalization of \cite[Prop. 3.3]{Simpson88} in the special case of VHSs. For smooth $Y$, the result is established in  \cite[Prop. 2.4]{VZ07}.

\begin{theorem}[Polystability of the associated system of log Hodge bundles]
	\label{thm:simpsonVHS_singular}
	Let $(Y,D)$ be a projective log canonical pair such that $K_Y + D$ is big and nef. Let $\bV$ be a variation of Hodge structures on $U:=Y\setminus D$ with unipotent local monodromies at infinity. Let $(\mathscr{E},\theta)$ be the log-Higgs sheaf obtained from the associated system of log-Hodge bundles $(E^\circ, \theta^\circ)$ on the regular part $(Y,D)_\reg = (Y^\circ, D^\circ)\hookrightarrow (Y,D)$. Then, the following statements hold:
	\begin{enumerate}
	    \item[(i)] For any proper saturated log-Higgs subsheaf $\mathscr{F} \subset \mathscr{E}$ we have \begin{equation}\label{eq:muinequ}\mu(\mathscr{F})\leq \mu(\mathscr{E})=0;\end{equation} i.e., the log-Higgs sheaf $(\mathscr{E},\theta)$ is semistable with respect to the polarisation $K_Y + D$.
	    \item[(ii)] If $K_Y + D$ is ample on $U$ and equality holds in \eqref{eq:muinequ}, there exists a holomorphic splitting \begin{equation}\label{eq:splitting}
	        \mathscr{E}|_{U_\reg}= \mathscr{F}|_{U_{\reg}}\oplus \mathscr{W}
	    \end{equation} of  systems of Hodge bundles over the regular locus $U_{\reg}$ of $U$, induced by a splitting of the corresponding variation of Hodge structures.
	    \item[(iii)] If moreover $(Y,D)$ is an orbifold pair, the  splitting  \eqref{eq:splitting} extends to $Y$, thus making $\mathscr{F}$ a system of orbifold log-Hodge bundles direct summand of $\mathscr{E}$. In particular, $\mathscr{F}|_{U_\reg}$ is locally free, and all orbifold Chern classes of $\mathscr{F}$ vanish.
	\end{enumerate}
\end{theorem}
\begin{proof}
	 We argue as in the proof of \cite[Prop.~5.1]{mochizuki_asterisque}.
	 We first show (i). Let $Y'$ be intersection of the regular locus of $(Y,D)$ and the set of points  where $\mathscr{E}$ is a vector bundle. Note that $Y\setminus Y'$ has codimension at least two in $Y$.

	Since $H=K_Y +  D$ is big and nef, it is semiample by the log canonical version of the Basepoint-Free-Theorem, see \cite{FujinoNonvanishing}, and so we can find a smooth complete intersection curve $C$ without loss of generality for $H$ (and not some multiple)  such that $C$ intersects only the smooth part of $D$ transversally and such that $C$ lies in $Y'$. 
	If $\mathscr{F}$ is a proper saturated sub-Higgs sheaf $\mathscr{F}\subset \mathscr{E}$, consider the intersection $Y''$ of $Y'$ with the locus where $\mathscr{F}$ is a sub vector bundle of $\mathscr{E}$, which again has complement of codimension at least two. Then, we may adapt our choice of the  smooth complete intersection curve $C$ for $H$ so that it additionally lies in $Y''$.

	Then, $(\mathscr{E},\theta)|_{C}=(E^\circ,\theta^\circ)|_{C}$ is a system of log-Hodge bundles   with unipotent local monodromy at infinity on $C$ and $\mu(\mathscr{E})=\deg_{C}(\mathscr{E}|_{C})=0$. The inequality $\mu(\mathscr{F})\leq \mu(\mathscr{E})=0$ hence follows from the curve case proven in  \cite[Sec.~6]{Simpsoncurves}.
	
	Now assume that $K_Y + D$ is ample on $U$ as in (ii).  The Hodge metric over $U$  yields an orthogonal decomposition 
	\begin{equation}\label{eq:smoothsplit}\mathscr{E}|_{U''}=\mathscr{F}|_{U''}\oplus \mathscr{W}''\end{equation} as smooth complex vector bundles, where $U'':=Y''\setminus (Y''\cap D)$. Consider \emph{any} smooth curve $C \subset Y''$. Owing to the result in the smooth case, e.g. see \cite[Prop. 2.4]{VZ07}, applied over the curve $C$, the orthogonal decomposition is holomorphic when restricted to $C\cap U''$ and there indeed induced by a decomposition of variation of Hodge structures. This, as in Mochizuki's proof of \cite[Prop.~5.1]{mochizuki_asterisque}, implies that the decomposition \eqref{eq:smoothsplit} is a holomorphic decomposition of system of Hodge bundles over $U''$ and, a fortiori, induced by a splitting in the category of harmonic bundles. Since the complement $U\setminus U''$ is of codimension greater than two, and since $\mathscr{E}$ is a vector bundle on $U$, by reflexivity we can extend the holomorphic decomposition to all of $U$, obtaining
	\begin{equation}\label{eq:split_on_U}
	    \mathscr{E}|_U = \mathscr{F}|_U \oplus \mathscr{W}_U. 
	\end{equation}
    In particular, we have a splitting on the regular locus of $U$. This also implies that $U''$ coincides with $U_{\reg}$,  since $\mathscr{F}$ is a direct summand of the vector bundle $\mathscr{E}$ over $U$. 
	
	Consider now the situation of (iii), so that $(Y,D)$ is an orbifold pair. Then, the VHSs on $U_{\reg}$ found in (b) that induce $\mathscr{F}|_{U_{\reg}}$ and $\mathscr{W}|_{U_{\reg}}$ can be extended to (orbifold) local systems on $U$ inducing the vector bundles $\mathscr{F}|_U$ and  $\mathscr{W}|_U$ everywhere on $U$. The direct sum of these (orbifold) local systems by construction is the local system underlying $\mathbb{V}$. Analogous to the smooth case, taking Deligne extensions as in \autoref{rem:orbifoldDeligneExtension} is functorial, so that direct sums go to direct sums. Therefore, the splitting \eqref{eq:split_on_U} extends over $D$ to give a splitting $\mathscr{E} = \mathscr{F} \oplus \hat{\mathscr{W}} $, i.e., making $\mathscr{F}$ a direct summand of $\mathscr{E}$. As the orbifold Chern classes of $\mathscr{E}$ vanish by \autoref{prop:deligneext}, the orbifold Chern classes of $\mathscr{F}$ likewise vanish. 
\end{proof}

The previous result plus the recent generalisations of semistability properties to the singular setup described in \autoref{subsect:stability} taken together are enough to derive the Arakelov inequality in our context, as we will see next.

\begin{proof}[Proof of \autoref{intro:inequality}]
We claim that the proof in the smooth setup given in \cite[Thm. 6.4]{viehwegsurvey} works also in our more general situation, in which the pair $(Y,D)$ is log-canonical and the polarization $K_Y+D$ is only big and nef. The argument indeed consists of semistability considerations that we have already seen to generalise to the singular setup; more precisely, the argument is combining the following inputs (references to the singular versions are given in parentheses):
\begin{enumerate}
    \item existence of Harder-Narasimhan filtrations for projective  log canonical pairs with $K_X + D$ big and nef, which is used as a polarisation to measure semistability (\autoref{thm:HN}),
    \item the fact that (reflexive) tensor products of semistable sheaves are semistable (\autoref{thm:tensor_poly}),
    \item  the semistability of  $\Omega_Y^{[1]}(\log D) = \mathscr{T}_Y(- \log D)^\vee$ (\autoref{prop:guenanciass}), and
    \item the semistability of the log-Higgs sheaf $(\mathscr{E}
    ,\theta)$ associated to $\bV$ (part (1) of \autoref{thm:simpsonVHS_singular}). 
\end{enumerate}
The proof hence generalises to our setup, establishing the Arakelov inequality claimed in \autoref{intro:inequality}. 
\end{proof}

\section{The case of equality}
\label{sec:proof}
In this section we prove \autoref{thm:intro1}.  For this, 
recall our setup from \autoref{subsec:setting}.	Let $U$ be an irreducible non-singular projective variety of dimension $d\geq 2$ or a quasi-projective surface and $f : A\to U$ a family of principally polarized Abelian varieties of dimension $g$ such that $R^1f_*\underline{\bC}$ has unipotent local monodromies at infinity. Furthermore, in the quasi-projective surface case we suppose that the induced map $\varphi_f:U \to \cA_g$ is proper (i.e., non-extendable, see \autoref{subsect:nonextendable}) and generically finite.  

Possibly replacing $U$ by a finite étale cover, we may assume without loss of generality that the period map $\varphi_f$ associated with $f$ actually maps to a \emph{fine} moduli space $\A_g$ of principally polarised Abelian varieties with a suitable level structure, see the discussion in \cite[Sec. 2.2] {MVZ}.

\subsection{Setup and first observations}\label{assumptionsineq} By \autoref{intro:fact}, the statement of \autoref{thm:intro1} is implied by the corresponding statement on the model $(Y_c, D_c)$ constructed in \autoref{sec:factorization}. In particular, note that  \autoref{prop:propernesstominimal} and \autoref{prop:propernesstomodel} ensure that properness of the period map is preserved when we pass to the model $(Y_c, D_c)$. From now we therefore assume that $(Y, D)$ is a projective variety with canonical singularities for $D=\emptyset$ or a log-canonical orbifold surface pair for $D \not =\emptyset$ such that $K_Y+D$ is big, nef and ample on $U=Y\setminus D$. Moreover, there is a generically finite proper period map $\varphi: U\ra \cA_g$ to some fine moduli space of principally polarised Abelian varieties, inducing a VHS  $R^1f_*\underline{\bC}$ on $U$ with unipotent local monodromies at infinity and with corresponding Higgs bundle $(G= G^{1,0} \oplus G^{0,1}, \tau^{\circ\circ})$ on $U$.  

We will denote by $(\mathscr{G}=\mathscr{G}^{1,0}\oplus \mathscr{G}^{0,1},\tau)$ the associated log-Higgs sheaf on $Y$.  Note that $\mathscr{G}$ is locally free if $D=\emptyset$, since then the VHS is defined on $U=Y$. If $D\neq \emptyset$ and $(Y,D)$ is an orbifold pair, then $G$ is an orbifold vector bundle, i.e., it defines and is induced by a vector bundle on the smooth DM-stack $\mathcal{Y}$ having coarse moduli space $Y$, see \autoref{sec:deligneext}. 

The Higgs field $\tau: \mathscr{G}^{1,0} \otimes \mathscr{T}_Y(-\log D) \ra \mathscr{G}^{0,1}$ induces the Kodaira-Spencer map 
\begin{gather}\label{KS}
  \tau_{\mathrm{KS}}: \mathscr{T}_Y(-\log D) \longrightarrow \Sym^{[2]}(G^{0,1}) \subset G^{0,1} [\otimes] (G^{1,0})^{\vee}. 
\end{gather}

Since $\A_g$ is a fine moduli space, it tangent sheaf $\mathscr{T}_{\A_g}$ is locally free, associated with the tangent bundle of $\A_g$ which we will denote by $T_{\A_g}$. We have natural identifications of vector bundles  $\Sym^{[2]}(\mathscr{G}^{0,1})|_U  = \Sym^2(G^{0,1}) \cong \varphi^*(T_{\A_g})$ over $U$. Restricting $\tau_{\mathrm{KS}}$ to $U$ and applying this identification to the right hand side of \eqref{KS},  we see that $\tau_{\mathrm{KS}}$ coincides with the differential $d\varphi$ of the period map $\varphi$ between the tangent sheaf $\mathscr{T}_U$ of the (potentially singular) variety $U$ and the pullback of the tangent bundle of $\A_g$. Since $\varphi$ is generically finite and $\mathscr{T}_Y(-\log D)$ is by definition torsion free, $\tau_{\mathrm{KS}}$ is a generically injective and  hence injective  morphism of sheaves.

\subsection{Strategy of the proof} The proof of  \autoref{thm:intro1} is organised as follows: We start by considering the decomposition of $R^1f_*\mathbb{C}_A$ into irreducible
$\mathbb{C}$-variations of Hodge structures over $U$. This induces a corresponding decomposition of the logarithmic Higgs bundle into log-Higgs subbundles. By assumption, each irreducible non-unitary factor $\mathbb{V}$ gives rise to a log-Higgs sheaf $(\mathscr{E}=\mathscr{E}^{1,0}\oplus \mathscr{E}^{0,1},\theta)$ satisfying the Arakelov equality. For each such non-unitary factor, we show in \autoref{prop:injective} that the stability assumption for all $\Sym^{[m]}\bigl(\Omega^{[1]}_{Y}(\log D)\bigr)$ together with the Arakelov equality implies that iterated Higgs fields are  injective. 
As a consequence, we obtain a length inequality  for each non-unitary factor in \autoref{prop:directfactor}.  In \autoref{prop:directI} we show that having equality in this inequality implies that the tangent sheaf $\mathscr{T}_Y$ is a direct summand of $\mathscr{H}\negthinspace om(\mathscr{E}^{1,0}, \mathscr{E}^{0,1})$ on $U$ for each irreducible non-unitary factor.

In the compact case, this already allows us to conclude the argument: having even a single non-unitary factor with this splitting property implies that $U=Y$ is smooth, so that we have reduced ourselves to the compact case of \cite{MVZ} and hence can conclude.

In the (two-dimensional) non-compact case, we proceed in a slightly different way, partly reversing the order of arguments in \cite{MVZ} and filling in some details: First, we show that over $U$ the injective differential of the period map splits not only holomorphically, but orthogonally with respect to the Hermitian metric associated with the Siegel metric on $\mathcal{A}_g$, see \autoref{prop:orthsplit}. On the level of universal covers, this implies  that the period map has a Hermitian symmetric subspace of Siegel space as its image, see \autoref{totgeoemb}. This subspace can only be a bidisc or the two-ball. The proof of  \autoref{thm:intro1} is completed by excluding the first case using the stability assumption on the symmetric powers of the log-cotangent bundle. The reader is refereed to \autoref{subsect:problemsinMVZ} for a comparison with the route taken in \cite[Sect.~7]{MVZ}.

\subsection{Analysis of iterated Higgs fields}

Let $\mathbb{V} \subset \mathbb{W}$ be a non-unitary irreducible
$\mathbb{C}$-sub-local system of $R^1f_*\underline{\bC}$, and let $(\mathscr{E}=\mathscr{E}^{1,0}\oplus \mathscr{E}^{0,1},\theta)$ be the associated system of log-Higgs sheaves on $Y$, as introduced in \autoref{sec:deligneext}. Let
\[(h,h'):=(\rk(\mathscr{E}^{1,0}), \rk(\mathscr{E}^{0,1}))\] and consider the saturated Higgs subsheaf $\langle \det(\mathscr{E}^{1,0})\rangle_s\subset \bigwedge^{[h]}(\mathscr{E},\theta)$
generated by $\det(\mathscr{E}^{1,0})$; this means that
\[\langle \det(\mathscr{E}^{1,0})\rangle_s^{h-m,m}=\Im_s\bigl(\theta^{(m)}\bigr)\]
is the saturated hull of the image of the iterated Higgs field $\theta^{(m)}$ as defined in \autoref{sec:deligneext}. Now note that the length defined in \autoref{def:lenght2} is unchanged if one replaces $\langle \det(\mathscr{E}^{1,0}) \rangle$ by its saturation. Indeed, both
coincide on a Zariski open dense subset, and the length depends only on the
generic non-vanishing of the iterated Higgs fields. Hence, we obtain 
\begin{equation}\label{eq:saturated_length}
\zeta(\mathscr{E},\theta)
=\max\{m\in \bN \mid \langle\det(\mathscr{E}^{1,0})\rangle_s^{h-m,m}\neq 0\}.
\end{equation}

The following observation (motivated by the discussion in \cite[Sect.~7.2]{MVZ})  concerning the behaviour of iterated Higgs fields in relation to length will be a key ingredient for the subsequent arguments. 

\begin{prop}[Injectivity of iterated Higgs fields]\label{prop:injective}
Under the assumptions made in \autoref{assumptionsineq}, let $\mathbb{V} \subset \mathbb{W}$ be a non-unitary irreducible $\mathbb{C}$-sub-local system of $R^1f_*\underline{\bC}_A$, and let $(\mathscr{E}=\mathscr{E}^{1,0}\oplus \mathscr{E}^{0,1},\theta)$ the associated Higgs sheaf on $Y$. 
If $\Sym^{[m]}\bigl(\Omega^{[1]}_{Y}(\log D)\bigr)$  is slope-stable with respect to $K_{Y} + D$ for all $m \geq 1$, then the Arakelov equality  $\mu(\mathscr{E}^{1,0})-\mu(\mathscr{E}^{0,1}) = \mu\bigl(\Omega^{[1]}_{Y}(\log D)\bigr)$ implies that the iterated Higgs field
\[
\theta^{(m)}: \det(\mathscr{E}^{1,0})[\otimes] \Sym^{[m]}\bigl(\mathscr{T}_Y(-\log D)\bigr) \, \longrightarrow \, \bigwedge^{[h-m]}\mathscr{E}^{1,0}\,[\otimes]\, \bigwedge^{[m]}\mathscr{E}^{0,1}
\]
is a morphism between semistable sheaves of the same slope. Moreover, $\theta^{(m)}$ is injective for all $m \le \zeta(E,\theta)$ .
\end{prop}
\begin{proof}
Recall from \autoref{intro:inequality} that $\mathscr{E}^{1,0}$ and $\mathscr{E}^{0,1}$ are both semistable. Moreover, $\mathscr{T}_Y(-\log D)$ is stable, and hence semistable by assumption. Since by \autoref{thm:tensor_poly} the reflexive tensor product of semistable sheaves is semistable, the sheaf $E^{1,0}[\otimes]\mathscr{T}_Y(-\log D)$
is semistable. Therefore,  the Arakelov equality implies that (the reflexive extension of) the Higgs field 
\[\theta: \mathscr{E}^{1,0}[\otimes]\mathscr{T}_Y(-\log D) \to \mathscr{E}^{0,1}\]
is a morphism between semistable sheaves of the same slope. Similarly, since reflexive tensor products, Homs, reflexive symmetric powers and reflexive wedge products of semistable sheaves is semistable by \autoref{thm:tensor_poly} and slope is additive under tensor product, the Arakelov equality implies 
\[\mu\bigl(\bigwedge^{[h-m]}\mathscr{E}^{1,0}[\otimes] \bigwedge^{[m]}\mathscr{E}^{0,1}\bigr)=m\mu\bigl(\mathscr{T}_Y(-\log D)\bigr)+\mu\bigl(\det(\mathscr{E}^{1,0})\bigr), \] and 
thus the iterated Higgs field $\theta^{(m)}$ is a morphism between semistable sheaves of the same slope as well. Moreover, by definition, it is nonzero if and only if $m \le \zeta(\mathscr{E},\theta)$. Set $\mathscr{F}_m:=\det(\mathscr{E}^{1,0}) [\otimes] \Sym^{[m]}\bigl(\mathscr{T}_Y(-\log D)\bigr)$ and consider the exact sequence
\begin{gather}\label{eq:exact}
  0 \longrightarrow \mathscr{K}_m \longrightarrow
\mathscr{F}_m
\overset{\theta^{(m)}}{\longrightarrow}
\langle \det(\mathscr{E}^{1,0})\rangle^{h-m,m} \longrightarrow 0, 
\end{gather}
where $\mathscr{K}_m := \ker(\theta^{(m)})$. We show that if $m \le \zeta(\mathscr{E},\theta)$, then $\mathscr{K}_m = 0$. Indeed, suppose that
$\mathscr{K}_m \neq 0$. Then $\mathscr{K}_m$ is a proper subsheaf of $\mathscr{F}_m$. By assumption $\Sym^{[m]}\bigl(\Omega^1_{Y}(\log D)\bigr)$  is stable for all $m$, so, as the same is true for $\mathscr{F}_m$, since $\det(\mathscr{E}^{1,0})$ has rank one. Hence,
$\mu(\mathscr{K}_m) < \mu(\mathscr{F}_m)$. From \autoref{eq:exact} we then infer that \[\mu(\langle \det(\mathscr{E}^{1,0})\rangle^{h-m,m}) > \mu(\mathscr{F}_m).\] This, however, contradicts the semistability of $\bigwedge^{[h-m]}E^{1,0}[\otimes] \bigwedge^{[m]}E^{0,1}$. We conclude that $\mathscr{K}_m=0$, so that  $\theta^{(m)}$ is an injective sheaf morphism for $m \le \zeta(E,\theta)$.
\end{proof}

\subsection{The length inequality}

Now the first step of the proof of \autoref{thm:intro1} is the following result.

\begin{prop}[Length (in)equality, cf.~{\cite[Lem.~7.2]{MVZ}}]\label{prop:directfactor}
Under the assumptions made in \autoref{assumptionsineq}, let $\mathbb{V}$ be a non-unitary irreducible $\mathbb{C}$-sub-local system of $R^1f_*\underline{\bC}_A$, and $(\mathscr{E}=\mathscr{E}^{1,0}\oplus \mathscr{E}^{0,1},\theta)$ the associated Higgs sheaf on $Y$. 
Assume furthermore that $\Sym^{[m]}\bigl(\Omega^1_{Y}(\log D)\bigr)$  is slope-stable with respect to $K_{Y} + D$ for all $m$ and that the Arakelov equality  $\mu(\mathscr{E}^{1,0})-\mu(\mathscr{E}^{0,1}) = \mu\left(\Omega^1_{Y}(\log D)\right)$ holds. Then, the following hold: 

\begin{enumerate}
    \item[(i)] The length fulfils the \emph{inequality}
	\begin{equation}\zeta(\mathscr{E},\theta)\geq\frac{hh'(d+1)}{(h+h')d} ,\end{equation}
	in which  equality is attained if and only if $\langle \det(E^{1,0})\rangle_s$ is a system of Hodge bundles direct summand of $\bigwedge^{[h]}(\mathscr{E},\theta)$ on $U_\reg$, where the latter one coincides with the system of Hodge bundles $\bigwedge^{h}(E^\circ,\theta^\circ)$.
	\item[(ii)]	If moreover $(Y,D)$ is an orbifold surface pair, then \begin{equation}\widehat{c}_1(\langle \det(E^{1,0})\rangle_s)=0 \;\;\text{ and } \;\; \widehat{c}_2(\langle \det(E^{1,0})\rangle_s)=0. \end{equation}
\end{enumerate}
\end{prop}

\medskip

\begin{proof}
For the proof, we set $\mathscr{F}^{h-m,m}:=\langle \det(\mathscr{E}^{1,0})\rangle_s^{h-m,m}$, 
so that $\langle \det(\mathscr{E}^{1,0})\rangle_s=\bigoplus_{m=0}^{\zeta(\mathscr{E},\theta)}\mathscr{F}^{h-m,m}$. By \autoref{lem:addendum}, we have
\begin{align}\label{slopeF}
\mu\bigl(\mathscr{F}^{h-m,m}\bigr)&=(h-m)\mu(\mathscr{E}^{1,0})+m\mu(\mathscr{E}^{0,1}). \end{align}
Since $\c_1(\mathscr{E}^{1,0})+\c_1(\mathscr{E}^{0,1})=0$, we have that $h\mu(\mathscr{E}^{1,0})+h'\mu(\mathscr{E}^{0,1})=0$. Thus, the Arakelov equality implies 
\begin{gather}\label{eq:arakelovbis}
  -\mu\bigl(\mathscr{T}_Y(-\log D)\bigr)=\mu(\mathscr{E}^{1,0})-\mu(\mathscr{E}^{0,1})=\frac{h+h'}{h'}\mu(\mathscr{E}^{1,0}).  
\end{gather}
Putting equations \autoref{slopeF} and \autoref{eq:arakelovbis} together we obtain 
 \begin{align*}
 	\mu\bigl(\mathscr{F}^{h-m,m} \bigr)&=m\mu\bigl(\mathscr{T}_Y(-\log D)\bigr)+h\mu(\mathscr{E}^{1,0})\\
 	&=\left(m-\frac{hh'}{h+h'}\right)\mu\bigl(\mathscr{T}_Y(-\log D)\bigr).
 \end{align*}
 Note in addition that 
 \[\mu(\langle \det(\mathscr{E}^{1,0})\rangle_s) =\frac{1}{\rk\left(\langle \det(\mathscr{E}^{1,0})\rangle_s\right)}\sum_{m=0}^{\zeta(\mathscr{E},\theta)} \mu\left(\mathscr{F}^{h-m,m} \right)\rk\left(\mathscr{F}^{h-m,m} \right).\]
 Moreover, from \autoref{prop:injective}, we obtain 
 $$\rk\left(\mathscr{F}^{h-m,m} \right)= \rk\left(\Sym^{[m]}(\mathscr{T}_Y(-\log D))\right)=\binom{d+m-1}{m}.$$
Since $\langle \det(\mathscr{E}^{1,0})\rangle_s$ is a saturated Higgs subsheaf of  $\bigwedge^{[h]}(\mathscr{E},\theta)$,  part (i) of \autoref{thm:simpsonVHS_singular}  implies that $\mu(\langle \det(\mathscr{E}^{1,0})\rangle_s)\leq 0$. Since $\Omega^{[1]}_Y(\log D)$ has positive degree, we see by putting the above observations together that
\begin{align*}
0 &\geq \frac{\mu(\langle \det(E^{1,0})\rangle_s)}{\mu(\Omega^1_Y(\log(D)))}=\frac{1}{\rk(\langle \det(E^{1,0})\rangle)_s}\sum_{m=0}^{\zeta(E,\theta)}\left(\frac{hh'}{h+h'}-m\right)\binom{d+m-1}{m}\\
&=\frac{\zeta(E,\theta)+1}{\rk(\langle \det(E^{1,0})\rangle)_s} \binom{\zeta(E,\theta)+d}{\zeta(E,\theta)+1}\left(\frac{hh'}{d(h+h')}-\frac{\zeta(E,\theta)}{d+1}\right)
\end{align*}
which gives the desired inequality by considering the last factor. 

In case of equality, part (ii) of \autoref{thm:simpsonVHS_singular} implies that $\langle \det(E^{1,0})\rangle_s$ is a direct factor of $\bigwedge^{[h]}(E,\theta)$ on $U_\reg$ in the category of system of Hodge bundles.

The statement made in part (ii)  about vanishing of the orbifold Chern classes follows from part (iii) of  \autoref{thm:simpsonVHS_singular}. 
\end{proof}

The next statement provides the main step for passing from the singular setup---necessarily considered here as a result of running the appropriate minimal model programs---to the smooth setup.  

\begin{prop}[Tangent is direct factor, cf.~{\cite[Prop.~4.2]{VZ07}}]\label{prop:directI}
In addition to the assumptions made in \autoref{assumptionsineq}, assume that $\Sym^{[m]}\bigl(\Omega^1_{Y}(\log D)\bigr)$  is slope-stable with respect to $K_{Y} + D$ for all $m$. Let $\mathbb{V}$ be a non-unitary irreducible $\mathbb{C}$-sub-local system of $R^1f_*\underline{\bC}_A$. Let $(\mathscr{E}=\mathscr{E}^{1,0}\oplus \mathscr{E}^{0,1},\theta)$ be the associated log-Higgs bundle on $Y$. If $(\mathscr{E}, \theta)$ satisfies both the Arakelov equality  $$\mu(\mathscr{E}^{1,0})-\mu(\mathscr{E}^{0,1}) = \mu\left(\Omega^{[1]}_{Y}(\log D)\right)$$ and the length equality $$\zeta(\mathscr{E},\theta)=\frac{hh'(d+1)}{(h+h')d},$$ then the tangent sheaf $\mathscr{T}_U = \mathscr{T}_Y(-\log D)|_U$ is a direct summand of the vector bundle $\Hom(E^{1,0}, E^{0,1})$ over $U$. The corresponding splitting is orthogonal for the Hermitian form induced by the Hodge metric over $U_\reg$. 
\end{prop}

\begin{proof}
By part (i) of \autoref{prop:directfactor}, on the regular locus of $U$ we have that $\langle \det(E^{1,0})\rangle_s$ is a  system of Hodge bundles summand of $\bigwedge^{h}(E,\theta)$. Therefore, in each of the $(p,q)$-pieces the corresponding piece of $\langle \det(E^{1,0})\rangle_s$ has holomorphic orthogonal complement. By reflexivity, we can extend the corresponding holomorphic decomposition to all of $U$. In particular, we have that the saturated image of  $\theta^{(1)}$, which is the $(h-1,1)$-piece of $\langle \det(E^{1,0})\rangle_s$,  is a direct summand of the $(h-1,1)$-piece of $\bigwedge^{h}(E,\theta)$, with holomorphic orthogonal complement over $U_\reg$.
 This implies  that the saturated image of
\[\tau_{\mathrm{KS}}: \theta^{(1)}\otimes \det(\mathscr{E}^{1,0})^\vee: \mathscr{T}_Y(-\log D)\to \mathscr{H}\negthinspace om(E^{1,0},E^{0,1})\]
 is a direct summand over $U$, with holomorphic orthogonal complement over $U_\reg$.

 Furthermore, since $\mathscr{T}_Y(-\log D)$  and $\mathscr{H}\negthinspace om(E^{1,0},E^{0,1})$ are $\mu$-semistable sheaves on $Y$ that by the Arakelov equality have the same slope,  \autoref{lem:addendum} applies to show that the saturated image of $\tau_{\mathrm{KS}}$ is isomorphic to $\mathscr{T}_U = \mathscr{T}_Y(-\log D)|_U$ over $U$. Since we already saw in the first step of the proof that over $U$ this saturated image is a direct summand of $\mathscr{H}\negthinspace om(\mathscr{E}^{1,0}, \mathscr{E}^{0,1})|_U = \Hom(E^{1,0},E^{0,1})$  with holomorphic orthogonal complement over $U_\reg$, the claim follows.
\end{proof}

\subsection{The compact case}

We now prove \autoref{thm:intro1} in case that $D = \emptyset$. This is achieved  by a reduction to the smooth case, see \autoref{prop:MVZnonef}. 

\begin{proof}[Proof of \autoref{thm:intro1} (compact case)] 
We now consider the case where $D=\emptyset$ and $Y=U$ has canonical singularities. We want to show that having both the Arakelov  and length \emph{equality} for each non-unitary summand in $R^1f_*\underline{\bC}_A$ implies that $\varphi(Y)$ is a totally geodesically embedded ball quotient. 

Let $\mathbb{V}$ be a such a non-unitary irreducible $\mathbb{C}$-sub-local system of $R^1f_*\underline{\bC}_A$, and $(E=E^{1,0}\oplus E^{0,1},\theta)$ the associated Higgs bundle on $Y$; note that in this case $\mathscr{E}= E$ is a vector bundle on $Y$, since the VHS is defined on the whole of $Y=U$. In particular, also $\mathscr{H}\negthinspace om(\mathscr{E}^{1,0}, \mathscr{E}^{0,1}) = \Hom(E^{1,0}, E^{0,1})$ is a vector bundle on $Y$. By \autoref{prop:directI} $\mathscr{T}_Y = \mathscr{T}_U$ is a direct factor of $\Hom(E^{1,0},E^{0,1})$. As a consequence, $\mathscr{T}_Y$ is locally free on everywhere on $Y$, which implies that $Y$ is smooth by the positive solution of the Zariski-Lipman conjecture for klt spaces, see \cite[Thm.~6.1]{GKKP} or \cite[Thm.~3.8]{Druel}. At this point, $Y=U$ satisfies the assumptions of the compact version of \cite[Thm.~0.2]{MVZ} discussed in \autoref{prop:MVZnonef}, so that the proof is complete. 
\end{proof}

\subsection{The non-compact case}\label{subsect:noncompact_case}

From now on we will focus on the case where $(Y, D)$ is a log-canonical orbifold surface pair and $D \not =\emptyset$.

\begin{prop}[Tangent is direct factor II, cf.~{\cite[Rem.~4.4]{VZ07}}]\label{prop:orthsplit}
The assumptions in \autoref{thm:intro1} imply that $U$ is smooth and the differential of the period map
\[
d\varphi: T_U=\mathscr{T}_Y(-\log D)|_U \hookrightarrow \varphi^*T_{\A_g}.
\]
splits holomorphically and orthogonally with respect to the (Hermitian metric induced by the)  Siegel metric. In particular, the period map $\varphi:U\ra \A_g$ is an immersion. 
\end{prop}
\begin{proof} 
Smoothness of $U$ follows as in the proof of the compact case given just above, so that the first important observation is that $U = U_\reg$.

Consider on $U$ the decomposition of $R^1f_*\mathbb{C}_A$ into irreducible
$\mathbb{C}$-variations of Hodge structures. For those factors $\mathbb{V}_i$ that are
not defined over $\mathbb{R}$, we denote by $\overline{\mathbb{V}}_i$ the respective
complex conjugate. After reordering the factors, as in \cite[Sec. 4]{VZ07} we write 
\begin{gather*}
R^1f_*\underline{\bC}_A=\mathbb{V}_1\oplus \overline{\mathbb{V}}_1\oplus \dots \oplus \mathbb{V}_{r'}\oplus \overline{\mathbb{V}}_{r'}\oplus\mathbb{V}_{r'+1} \oplus \dots \oplus \mathbb{V}_{r},
\end{gather*}
where $\mathbb{V}_i$ is defined over $\mathbb{R}$ for $r'+1 \leq i \leq r$. Let $(E_i^{1,0}\oplus E_i^{0,1},\theta_i)$ denote the Higgs bundle of $\mathbb{V}_i$. With the notations in \autoref{assumptionsineq}, we have that, for $i\leq r'$, the bundles $E_i^{0,1}$ and $(E_i^{1,0})^\vee$ occur as direct summands of $\mathrm{Sym}^2(\mathscr{G}^{0,1})|_U=\varphi^*(\mathscr{T}_{\A_g})$, and they intersect trivially. Up to renumbering, we can also assume that
\begin{gather*}
    \rk(E_i^{1,0}) \leq \rk(E_i^{0,1}) \quad \text{for}\ i\leq r'
\end{gather*}
and
\begin{gather*}
    \rk(E_i^{1,0})=\rk(E_i^{0,1}) \quad \text{for}\ r'+1\leq i\leq r.
\end{gather*}
Using this, we define the sheaf $H=\bigoplus_{i=1}^r H_i $ on $U$ by
\begin{gather*}
H_i:= \begin{cases}E_i^{0,1} \otimes (E_i^{1,0})^\vee & \text{if } i\leq r' \\
\mathrm{Sym}^2(E_i^{0,1}) & \text{if } r'+1\leq i\leq r.
\end{cases}
\end{gather*}

By construction, $H$ is a direct summand of $\mathrm{Sym}^2(G^{0,1})$ with holomorphic orthogonal complement, and the Kodaira-Spencer map \eqref{KS} maps into $H$. If $\mathbb{V}_i$ is non-unitary, the composite map
\begin{gather*}
  T_U \overset{\text{KS}}{\longrightarrow} H =\bigoplus_{i=1}^r H_i \overset{pr_i}{\longrightarrow} H_i.
\end{gather*}
is non-zero. By \autoref{prop:directI}, $T_U$ is a direct summand of $H_i$, with holomorphic orthogonal complement. On the other hand, the Kodaira-Spencer map vanishes on the
unitary factors. Therefore, using the splittings for all non-unitary
summands simultaneously, we obtain that $T_U$ is a direct summand of $H$ with holomorphic orthogonal complement. Since $H$ is itself a direct summand of $\mathrm{Sym}^2(G^{0,1})$ with holomorphic orthogonal complement, it follows that
$T_U$ is a direct summand of $\mathrm{Sym}^2(G^{0,1}) \cong \varphi^*T_{\A_g}$, whose orthogonal complement with respect to the natural metric on $\mathrm{Sym}^2(G^{0,1})$ is a holomorphic subbundle.

This implies the corresponding statement with respect to the Hermitian form associated with the Siegel metric on $\varphi^*T_{\A_g}$, since up to a scalar, it coincides with the Hodge metric on its tautological variation of Hodge structure, see \cite[Sect.~2]{FredianiColombo}. In particular,  the differential of the period map is injective at each point, and thus $\varphi$ is an immersion. 
\end{proof}

The proof of the following statement is a greatly expanded version of an argument that in \cite{MVZ} is attributed to Mok; in our approach, it plays a much more central role than in \emph{loc.~cit.}, as we obtain that the image of the period map is Hermitian symmetric directly, without a reduction to the equality case of the Miyaoka-Yau inequality. 

\begin{prop}[Totally geodesic embedding, cf.~{\cite[Claim~7.14]{MVZ}}]\label{totgeoemb}
Let $\tilde\varphi: \tilde U\ra \sieg_g$ be the lift of the period map from the universal cover $\tilde U$ of $U$ to Siegel upper half-space. 
The assumptions in \autoref{thm:intro1} imply that $\tilde\varphi: \tilde U\ra \sieg_g$ is a proper totally geodesic embedding. In particular, $\tilde\varphi(\tilde U)$ is a 2-dimensional hermitian symmetric subspace of $\sieg_g$.
\end{prop}

\begin{proof} 
Let $\Gamma$ be the image of the monodromy representation associated with the period map $\varphi$. More explicitly, for some choice of base point $x\in U$, $\Gamma$ is the image of the representation $\pi_1(U,x)\ra \operatorname{Sp}(2g, \mathbb{Z})$ associated with the integral local system $R^1f_*\underline{\mathbb{Z}}$ of the family $f$.  Recall that, by assumption, $\varphi$ maps into a fine moduli space $\A_g$ of principally polarized Abelian varieties with a suitable $N$-level structure $\Gamma(N):=\ker\left(\operatorname{Sp}(2g,\mathbb{Z})\ra \operatorname{Sp}(2g, \mathbb{Z}/N\mathbb{Z})\right)$ with $N\geq 3$. It follows that $\Gamma\subset \Gamma(N)$ and, in particular, that $\Gamma$ is torsion free, as $\Gamma(N)$ is. The period map factors as
\begin{gather*}
   \varphi:  U\overset{\varphi_{\Gamma}}{\longrightarrow} \Gamma\backslash\sieg_g \rightarrow \A_g.
\end{gather*}
By our assumptions and by \autoref{prop:orthsplit}, $\varphi$ is a proper immersion. Since $\Gamma\backslash\sieg_g \rightarrow \A_g$ is a covering map,  $\varphi_{\Gamma}$ is a proper immersion as well. Let $U_{\Gamma}$ be the quotient of the universal cover $\tilde U$ by the kernel of the monodromy representation $\pi_1(U,x)\ra \Gamma$. In other words,  $U_{\Gamma}$ is  the minimal covering of $U$ on which the VHS is trivialised.  The period map $\varphi_{\Gamma}$ lifts to $U_{\Gamma}$ and we obtain the following commutative diagram with Cartesian central square:

\begin{equation*}
\begin{tikzcd}
\tilde U\arrow[r, " "]\arrow[dr, "\tilde \varphi"'] & U_{\Gamma}  \arrow[r, " "] \arrow[d, "\tilde \varphi_{\Gamma}"']\arrow[dr, phantom, "\ulcorner", very near start]
& U \arrow[d, "\varphi_{\Gamma}" ] \arrow[dr, "\varphi"]& \\
& \sieg_g \arrow[r, "\pi_{\Gamma}"] & \Gamma\backslash\sieg_g\arrow[r, " "] & \A_g.
\end{tikzcd}
\end{equation*}
In other words, the lift $\tilde\varphi_{\Gamma}$ is the base-change of $\varphi_{\Gamma}$ along the covering map $\pi_{\Gamma}$. Since $\varphi_{\Gamma}$ is proper and properness is preserved under base-change, it follows that $\tilde\varphi_{\Gamma}$ is proper as well, cf.~\cite[Prop.~3.4]{BrunebarbeShafarevich}.  In particular, the image $\tilde\varphi_{\Gamma}(U_{\Gamma})$ is a closed analytic subset in $\sieg_g$. We want to show that it is in fact a totally geodesic (complex) submanifold. 

Since $\tilde\varphi_{\Gamma}$ is an immersion with closed image, it is a local embedding. In the following, we work locally on $U_\Gamma$ and think of $\tilde\varphi_{\Gamma}(U_{\Gamma})\subset\sieg_g$ as a closed embedded submanifold. We consider the exact sequence of holomorphic tangent bundles 
\begin{gather}\label{eq:tangentexact}
    0\ra T_{U_{\Gamma}} \ra \tilde\varphi_{\Gamma}^*T_{\sieg_g} \overset{\pi}{\ra} N_{\tilde\varphi} \ra 0,
\end{gather}
where $N_{\tilde\varphi_{\Gamma}}=\tilde\varphi_{\Gamma}^*T_{\sieg_g}/T_{U_{\Gamma}}$ denotes the normal bundle. Let
\[
\operatorname{II}^{2,0}: \Sym^2T_{U_{\Gamma}} \ra N_{\tilde\varphi_{\Gamma}}
\]
denote the $(2,0)$-part of the $\mathbb{C}$-linear extension of the second fundamental form $\operatorname{II}$ of $\tilde\varphi_{\Gamma}(U_{\Gamma})\subset \sieg_g$ with respect to the Siegel metric. Thus, for $x\in U_{\Gamma}$ and $r,s$ local sections of $T_{U_{\Gamma}}$ near $x$, we have
\[
\operatorname{II}^{2,0}_x(r \odot s) =\pi (\nabla_r s)_x,
\]
where $\nabla$ is the $\mathbb{C}$-linear extension of the Levi-Civita connection of the Siegel metric. We want to show that, under the assumptions of  \autoref{thm:intro1}, we have that $$\operatorname{II}^{2,0}_x\equiv 0$$ 
at all points $x\in U_{\Gamma}$. Note that this implies that the (real) Riemannian second fundamental form $\operatorname{II}$ vanishes at each point, i.e., that $\tilde\varphi_{\Gamma}$ is a totally geodesic immersion. Indeed, the Riemannian second
fundamental form $\operatorname{II}$ of a complex submanifold of a K\"ahler manifold is
compatible with the complex structure, see e.g.~\cite[Prop.~IX.9.1]{kobayashi-nomizu-II}. After complexification, this implies that $\operatorname{II}_{\mathbb C}$ has only components of type $(2,0)$ and $(0,2)$. Moreover, since $\operatorname{II}$ is real, the vanishing of the $(2,0)$ part also implies the vanishing of the $(0,2)$ part. 

By \autoref{prop:orthsplit}, the inclusion $T_U \hookrightarrow \varphi^*T_{\A_g}$ splits holomorphically and orthogonally with respect to the Hermitian metric induced by the Siegel metric. Passing to the covering $U_{\Gamma}$, we obtain a holomorphic splitting 
\begin{gather}\label{eq:split}
\tilde \varphi_{\Gamma}^*T_{\sieg_g}\cong T_{U_{\Gamma}}\oplus N    
\end{gather}
where $N$ denotes the orthogonal complement of $T_{U_{\Gamma}}$ in $\tilde \varphi_{\Gamma}^*T_{\sieg_g}$ with respect to the Siegel metric. Hence, $N$ is a holomorphic vector bundle and,  owing to the exact sequence \eqref{eq:split}, $N\cong N_{\tilde\varphi_{\Gamma}}$.  Now recall that, since $\sieg_g$ is K\"ahler, the Levi-Civita connection of the
Siegel metric induces on $T_{\sieg_g}\cong T^{1,0}\sieg_g$ the Chern connection
with respect to the Hermitian metric induced by the Siegel metric, see e.g.
\cite[Prop.~4.A.9]{huybrechts}.  Since the splitting in \eqref{eq:split} is both orthogonal and holomorphic, both the holomorphic structure as well as the Hermitian metric on  $\tilde\varphi_{\Gamma}^*T_{\sieg_g}$ are compatible with the decomposition. It follows that the Chern connection $\tilde\varphi_{\Gamma}^*T_{\sieg_g}$ is the direct sum of the Chern connections of $T_{\tilde U}$ and $N$. In particular, $\nabla$ preserves $T_{\tilde U}$, i.e., for all local sections $r,s$ of $T_{U_\Gamma}$ we have that $(\nabla_r s)$ is again a section of $T_U$, proving the vanishing of the second fundamental form. 

We thus have that $\tilde\varphi_{\Gamma}$ is a proper totally geodesic immersion, and we want to show that it is injective. For this, we consider on $U_{\Gamma}$ the pull-back of the Siegel metric via $\tilde\varphi_{\Gamma}$. Since the Siegel metric on $\sieg_g$ is geodesically complete and  $\tilde\varphi_{\Gamma}$ is proper, we have that the pull-back metric on $U_{\Gamma}$ is also geodesically complete. Indeed, under these assumptions one can directly check the convergence of Cauchy
sequences in $U_{\Gamma}$, showing that the pull-back metric on $U_{\Gamma}$ is
metrically complete. Together with the Hopf-Rinow theorem, this shows geodesic completeness.\footnote{See e.g.~\cite[Thm.~2.8, Ch.~7]{doCarmo} and
\cite[Cor.~2.10, Ch.~7]{doCarmo} for the analogous statement for closed embedded
submanifolds of complete Riemannian manifolds (the same argument applies in
our setting, using properness)}.  Now, assume that $x,y\in U_{\Gamma}$ satisfy $\tilde\varphi_{\Gamma}(x)=\tilde\varphi_{\Gamma}(y)$, and let $\gamma$ be a geodesic in $U_{\Gamma}$ between $x$ and $y$. Since $\tilde\varphi_{\Gamma}$ is totally geodesic, we have that $\tilde\varphi_{\Gamma}\circ \gamma$ is a geodesic in $\sieg_g$ as well, see \cite[Prop.~2.9, Ch.~6]{doCarmo}. Moreover, since $x$ and $y$ are assumed to have the same image in $\sieg_g$, $\tilde\varphi_{\Gamma}\circ \gamma$ has the same endpoints. On the other hand, $\sieg_g$ is a Hermitian symmetric space of non-compact type and, in particular, it is complete, simply connected, and has non-positive sectional curvature, see \cite[Prop.~6.3, Prop.~6.39, Prop.~6.45]{ziller}. By the Cartan-Hadamard theorem
\cite[Thm.~3.1, Ch.~7]{doCarmo}, the exponential map at every point is a
diffeomorphism. Hence, $\sieg_g$ has no non-constant closed geodesics. Hence $\tilde\varphi_{\Gamma}\circ \gamma$ is constant and, since $\tilde\varphi_{\Gamma}$ is an immersion, it follows that $\gamma$ is constant as well and thus that $x=y$. This shows that $\tilde \varphi_{\Gamma}$ is a proper totally geodesic embedding. 

Finally, we show that $\tilde\varphi_{\Gamma}(U_{\Gamma})$ is globally a symmetric subspace of Siegel space. To do so, we show it is preserved by the geodesic reflections of $\sieg_g$ at any point in $\tilde\varphi_{\Gamma}(U_{\Gamma})$, see \cite[Prop.~6.28]{ziller}. Let $s_x: \sieg_g\ra \sieg_g$ be the symmetry at the point $x\in \tilde\varphi_{\Gamma}(U_{\Gamma})\subset \sieg_g$.  For a point $y\in \tilde\varphi_{\Gamma}(U_{\Gamma})$, let $\gamma$ be a geodesic in $\tilde\varphi_{\Gamma}(U_{\Gamma})$ between $x$ and $y$. Since $\tilde\varphi_{\Gamma}(U_{\Gamma})$ is totally geodesic, $\gamma$ is also a geodesic in $\sieg_g$. Now $s_x$ flips the geodesic through $x$, thus $s_x(\gamma(t))=\gamma(-t)$ and, since $\tilde\varphi_{\Gamma}(U_{\Gamma})$ is geodesically complete, this implies that $s_x(y)\in \tilde\varphi_{\Gamma}(U_{\Gamma})$. Thus $s_x(\tilde\varphi_{\Gamma}(U_{\Gamma}))=\tilde\varphi_{\Gamma}(U_{\Gamma})$ and therefore $\tilde\varphi_{\Gamma}(U_{\Gamma})$ is a symmetric subspace of Siegel space. Thus, it is itself a Hermitian symmetric space of non-compact type. In particular, it is simply connected. Thus $U_{\Gamma}=\tilde U$ and $\tilde\varphi_{\Gamma}=\tilde\varphi$ is a proper totally geodesic embedding. In particular, the period map $\varphi_{\Gamma}$ into $\Gamma\backslash\sieg_g $ is an embedding.
\end{proof}

\begin{proof}[Proof of \autoref{thm:intro1} (non-compact case)]
We showed in \autoref{totgeoemb} that the period map 
$$\varphi_{\Gamma}: U \ra \Gamma\backslash\sieg_g$$ 
where $\Gamma$ is the monodromy of the period map $\varphi$, is a totally geodesic embedding. Let $Z$ be the image of $\varphi_{\Gamma}$. The symmetric space $X$ uniformising $Z$, i.e., the image of the lift $\tilde \varphi$ into $\sieg_g$, is a $2$-dimensional symmetric subspace of Siegel space, and hence a Hermitian symmetric space of non-compact type. By the classification of Hermitian symmetric spaces of non-compact type \cite[Table~5,p.~518]{Helgason}, it follows that either 
\begin{gather}\label{eq:2dimsymm}
   X\cong \mathbb{B}_2\quad \quad \text{or}\quad \quad X\cong \mathbb{B}_1 \times \mathbb{B}_1, 
\end{gather}
i.e., either $X$ is a $2$-dimensional ball or is the product of two one-dimensional discs. In the following we show that under the assumptions of \autoref{thm:intro1} only the first possibility can occur.

Since the period map $\varphi$ is proper and algebraic, its image $\varphi(U)=Z= \pi(X)$ is a closed subvariety in $\mathcal{A}_g$ to which we may apply \autoref{prop:mumford}. If $X\cong \mathbb{B}_1 \times \mathbb{B}_1$, the resulting Mumford compactification $(Y', D')$ of $Z$
has the property that $\Sym^{m}\Omega^1_{Y'}(\log D')$ is not stable for all $m\geq 1$ with respect to $K_{Y'}+D'$; more specifically, for all $m\geq 1$ the line bundle 
\begin{gather}\label{eq:destabilise}
  \mathscr{O}\bigl(m(K_{Y'}+D')\bigr) \hookrightarrow \Sym^{2m}\Omega^1_{Y'}(\log D')
\end{gather}
is a ``destabilising'' direct summand of $\Sym^{2m}\Omega^1_{Y'}(\log D')$. Now, choose $m$ such that $m(K_Y+D)$ is Cartier. We are aiming to show the following.

\medskip

\noindent
\textbf{Main Claim:} \emph{$\Sym^{[2m]}\Omega^{[1]}_Y(\log D)$ is not stable with respect to $K_Y+D$.}

\medskip

\noindent 
This will contradict our assumption \eqref{eq:stab_assump} and will show that the second possibility in \eqref{eq:2dimsymm} does not occur. Hence, $X$ is a ball and $U\cong Z$ is a ball quotient, thus concluding the proof.

To establish the Main Claim we compare the two compactifications $(Y,D)$ and $(Y', D')$ of $U\cong Z$. Recall that $(Y,D)$ is a log-canonical orbifold pair such that $U=Y\setminus D$ and $K_Y+D$ is big and nef, see \autoref{assumptionsineq}, while $(Y',D')$ is a smooth snc compactification of $U$ such that $K_{Y'}+D'$ is big and nef, see \autoref{prop:mumford}. To compare them, we consider a birational blow-up which is compatible with the two compactifications, i.e., we choose a smooth projective variety $\bar Y$ and a normal crossing divisor $\bar D\subset \bar Y$ such that $\bar Y\setminus \bar D=U$, the inclusion $i_1: U\hookrightarrow Y$ extends to a birational morphism $\sigma: \bar Y \ra Y$, and the inclusion $i_2: U\hookrightarrow Y'$ extends to a birational morphism $\bar \delta: \bar Y \ra Y'$ 
\begin{equation*}
\begin{tikzcd}
& (\bar Y, \bar D) \arrow[dl, "\sigma" '] \arrow[dr, "\bar\delta"] & \\
(Y,D) & & (Y', D').
\end{tikzcd}
\end{equation*}
Since $(Y,D)$ is log-canonical and since $\sigma$ is a log-resolution of the pair, the discrepancy formula \eqref{eq:discrepancy} yields
\begin{gather*}
    	K_{\bar Y} + \sigma_*^{-1}(D) \;\sim_{\bQ} \; \sigma^*(K_Y+D)+ \sum_{i}a_i E_i,
\end{gather*}
where $\sigma_*^{-1}(D)$ is the strict transform, the $E_i\subset \bar Y$ are $\sigma$-exceptional, and the coefficients $a_i\geq - 1$. It follows that
\begin{gather*}
    	K_{\bar Y} + \bar D \;\sim_{\bQ} \; \sigma^*(K_Y+D)+ \sum_{i}(a_i+1) E_i.
\end{gather*}
Setting $L:=\sigma^*(K_Y+D)$, and $E:=\sum_{i}(a_i+1) E_i$, we thus have that $E$ is effective, and there exists a meromorphic function $f\neq 0$ on $\bar Y$ such that 
\begin{gather*}
    K_{\bar Y} + \bar D =L+E+\frac{1}{m} \operatorname{div}(f).
\end{gather*}
Since $K_Y+D$ is nef, $L$ is nef. In particular, $P:=L+\frac{1}{m}\operatorname{div}(f)$ is nef as well. Summing up we have that 
\begin{gather}\label{eq:Zariski1}
    	K_{\bar Y} + \bar D = P+ E
\end{gather}
where $P$ is nef, and $E$ is effective with negative definite intersection matrix. Moreover, for each irreducible component $C$ of $E$ we have $P. C=0$ as $C$ is exceptional, $L=\sigma^*(K_Y+D)$, and $\operatorname{div}(f)$ is numerically trivial. In other words, \eqref{eq:Zariski1} is the Zariski decomposition of $K_{\bar Y} + \bar D$, see e.g.~\cite[Thm.~1.13.22]{Kawamata2024}. Similarly, the canonical bundle formula for $\bar\delta$ yields 
\begin{gather*}
    	K_{\bar Y} + \bar D = \bar\delta^*(K_{Y'}+D') + E' + \operatorname{div}(g)
\end{gather*}
where $E'$ is effective and exceptional, and $g\neq 0$. Setting $L:=\bar\delta^*(K_{Y'}+D')$ and $P':=L'+\operatorname{div}(g)$ we get
\begin{gather}\label{eq:Zariski2}
    	K_{\bar Y} + \bar D = P'+ E'
\end{gather}
where again $P'$ is nef, $E'$ is effective with negative definite intersection for and for each irreducible component $C$ of $E'$ we have $P' .C=0$. Thus \eqref{eq:Zariski2} is also the Zariski decomposition of $K_{\bar Y} + \bar D$. The uniqueness of such a decomposition then implies that $P=P'$. In particular, $mL'+\operatorname{div}(g)=mL+m\operatorname{div}(f)$, yielding that
\begin{equation}\label{eq:pullbackequality}
  \mathscr{O}_{\bar Y} (mL') =  \bar \delta^*\mathscr{O}_{Y'}(m(K_{Y'}+D')) \cong \sigma^*\mathscr{O}_{ Y}(m(K_Y+D)) = \mathscr{O}_{\bar Y}(mL).
\end{equation}

Our goal is to use the injection  \eqref{eq:destabilise} to obtain a corresponding subsheaf in $\Sym^{[2m]}\Omega^{[1]}_Y(\log D)$ that violates stability. Using \eqref{eq:pullbackequality}, the pullback of $\iota$ via $\bar\delta$ yields a sheaf morphism $\psi$ as follows. 
\begin{equation*}
\begin{tikzcd}
 \mathscr{O}_{\bar Y}(mL) \cong \mathscr{O}_{\bar Y}(mL') \arrow[r, "\bar\delta^*\iota"] \arrow[dr, "\psi"'] & \bar\delta^*\Sym^{2m}\Omega^1_{Y'}(\log D') \arrow[d, ""] \\
& \Sym^{2m}\Omega^1_{\bar Y}(\log \bar D).
\end{tikzcd}
\end{equation*}
Since it is generically non zero, $\psi$ is injective. Let $W\subset \bar Y$ be the complement of the $\sigma$-exceptional set in $\bar Y$. Restricting $\psi$ to $W$ we arrive at the following diagram of sheaf morphisms
\begin{equation*}
\begin{tikzcd}
\mathscr{O}_{\bar Y}(mL)|_W \arrow[r, hook] \arrow[d, equal]
& \Sym^{2m}\Omega^1_{\bar Y}(\log \bar D)|_W \arrow[d, equal ] \\
 \sigma^*\bigl(\mathscr{O}_{Y'}(m(K_{Y'+D'}))\bigr)|_W \arrow[r, hook] & \sigma^*\bigl(\Sym^{[2m]}\Omega^{[1]}_Y(\log D)\bigr)|_W.
\end{tikzcd}
\end{equation*}
Pushing this forward to $Y$ we obtain  an injection 
\begin{gather*}
    \psi_0':\mathscr{O}_Y(m(K_Y+D))|_{Y\setminus \sigma(Y\setminus W)}\hookrightarrow  \Sym^{[2m]}\Omega^{[1]}_Y(\log D)|_{Y\setminus \sigma(Y\setminus W)}.
\end{gather*}
Since $\sigma(Y\setminus W)$ is a finite set of points, and the sheaves on both sides are reflexive, $\psi_0'$ extends to an injective morphism of sheaves $\psi':\mathscr{O}_Y(m(K_Y+D))\hookrightarrow  \Sym^{[2m]}\Omega^{[1]}_Y(\log D).$ Since $\mathscr{O}_Y(m(K_Y+D))$ and $\Sym^{[2m]}\Omega^{[1]}_Y(\log D)$ by construction have the same $(K_Y+D)$-slope, the latter sheaf is not stable, which establishes the Main Claim and concludes the proof.
\end{proof}

\subsection{Relation to approach via MY-equality}\label{subsect:comparison_our_approach_with_MVZ}

In the proof above, after observing smoothness of $U$ we showed that the period map lifts to a totally geodesic embedding $\tilde\varphi: \tilde U\ra \sieg_g$ from the universal cover of $U$ to Siegel upper-half space, and then used the classification of two-dimensional symmetric spaces to single out the ball case. In this last section, we wish to compare this approach to the one adopted in \cite{MVZ}. Here, the strategy of the proof rather first aims to show that the Arakelov and length equalities imply that the base $(Y,D)$ of the family satisfies equality in the Miyaoka-Yau inequality, thus showing that $U$ is a ball quotient; total geodesy of the image is then shown in a second step.

We wish to expand on this, focusing on the case where $(Y, D)$ is a log-canonical orbifold surface pair, and comparing the compact with the non-compact case. For this, we will go through the argument of \cite[Claim 7.12]{MVZ}, which rests on showing the following: \\

\noindent \textbf{Key claim for ball quotient uniformisation via MY: }\emph{
In addition to the assumptions made in \autoref{assumptionsineq}, assume that $\Sym^{[m]}\bigl(\Omega^1_{Y}(\log D)\bigr)$  is slope-stable with respect to $K_{Y} + D$ for all $m$. Let $\mathbb{V}$ be a non-unitary irreducible $\mathbb{C}$-sub-local system of $R^1f_*\underline{\bC}_A$, and let $(\mathscr{E}=\mathscr{E}^{1,0}\oplus \mathscr{E}^{0,1},\theta)$ be the associated log-Higgs bundle on $Y$. If $(\mathscr{E}, \theta)$ satisfies both the Arakelov equality  
and the length equality 
, then $U$ is a ball quotient.}\\

The starting point to argue in this direction is the fact that, by part (i) of \autoref{prop:directfactor}, the Arakelov and length equalities imply that $\langle \det(\mathscr{E}^{1,0})\rangle_s$ is a system of Hodge bundles direct factor of $\bigwedge^{[h]}(\mathscr{E},\theta)$ on $U_{\reg}$.  Moreover, by part (ii) of the same result, we have $\widehat{c}_1(\langle \det(\mathscr{E}^{1,0})\rangle_s)=\widehat{c}_2(\langle \det(\mathscr{E}^{1,0})\rangle_s)=0$. The goal would now be to manipulate the vanishing of these Chern classes to get a Miyaoka-Yau equality.

As in \cite[Claim 7.12]{MVZ}, we consider for simplicity the case where there exists an orbifold line bundle $\mathscr{L}$ such that $\mathscr{L}^{\otimes [h]}\simeq \det(\mathscr{E}^{1,0}).$
Set $\mathscr{T}:=\mathscr{T}_Y(-\log D)$
and consider the log-Higgs sheaf
$$
(\mathscr{F}=\mathscr{F}^{1,0}\oplus \mathscr{F}^{0,1}=\mathscr{L}\oplus (\mathscr{L}[\otimes] \mathscr{T}),\  \mathscr{L} [\otimes] \mathscr{T}_Y(-\log D)  \ra \mathscr{L}[\otimes] \mathscr {T} ).
$$
The $h$-th symmetric power $\Sym^{[h]}(\mathscr{F})$ has $(h-m,m)$-piece given by $$\Sym^{[h-m]}\mathscr{L}[\otimes] \Sym^{[m]}( \mathscr{L}\otimes \mathscr{T}) \simeq
\det(E^{1,0})\otimes \Sym^{[m]} \mathscr{T}.$$ Thus, $\Sym^{[h]}(F)$ is isomorphic to
$\langle \det(\mathscr{E}^{1,0})\rangle$. Now, by \autoref{lem:addendum}, we know that $\langle \det(\mathscr{E}^{1,0})\rangle$ is isomorphic to  $\langle \det(\mathscr{E}^{1,0})\rangle_s$ on $U$.

\subsubsection{The compact case} When $U=Y$, we thus have that $\Sym^{[h]}(\mathscr{F})$ is isomorphic to
$\langle \det(E^{1,0})\rangle_s$. Hence, the vanishing of the first and second Chern classes of $\langle \det(E^{1,0})\rangle_s$ implies that $\widehat c_1(\Sym^{[h]}(\mathscr{F}))=\widehat c_2(\Sym^{[h]}(\mathscr{F}))=0.$ Hence, the discriminant $\widehat\Delta(\mathrm{Sym}^h(\mathscr{F}))$ also vanishes. 
By \cite[Lemma 3.3]{VZ07}, the discriminant of the symmetric power $\widehat\Delta(\Sym^{[h]}(\mathscr{F}))=0$ if and only if the $\widehat\Delta(\mathscr{F})=0$. On the other hand, $\mathscr{F}=\mathscr{L}\otimes(\mathscr{O}_Y\oplus \mathscr{T}),$ and the discriminant is invariant under tensoring with a line bundle. Hence $0=\widehat\Delta(\mathscr{F})=\widehat\Delta(\mathscr {O}_Y\oplus \mathscr{T}).$ We obtain
$3 \cdot \widehat{c}_2({T_Y})- \widehat{c}_1({T_Y})^2 =0$. The Miyaoka-Yau uniformization theorem applies, and $U=Y$ is uniformised by the two-dimensional complex ball. This approach also works in higher dimensions, leading to \autoref{prop:MVZnonef}.

\subsubsection{The non-compact case}\label{subsect:problemsinMVZ} When $D\neq \emptyset$, one would like to use the same strategy. Clearly, a key step in the argument above is the vanishing of the Chern classes of $\Sym^{[h]}(\mathscr{F})$, deduced from the vanishing of the Chern classes of$\langle \det(\mathscr{E}^{1,0})\rangle_s$. At this point, it is not clear to us how to do this, as $\Sym^{[h]}(\mathscr{F})$ is a priori just isomorphic to $\langle \det(\mathscr{E}^{1,0})\rangle$, which coincides with its saturation $\langle \det(\mathscr{E}^{1,0})\rangle_s$ only on $U$, cf.~\autoref{lem:addendum}, or, more precisely, on the open subset $U'$ of $Y$ where $\omega_Y(D)$ is ample, cf.~\cite[Addendum 5.2]{MVZ}.  The problem seems to already be present in similar arguments carried out in \cite{VZ07}, and is slightly hidden by a change of notation that suppresses the difference between $\langle \det(\mathscr{E}^{1,0})\rangle$ and its saturation, see just before \cite[Lemma 7.2]{MVZ}.

\vspace{0.4cm}

\begin{small}
\printbibliography
\end{small}

\ \\
\bigskip
\ \\

\begin{center}
    --------------------------------------------------------------------------------------
\end{center}

\ \\
\bigskip
\ \\

\end{document}